\documentclass[journal,twoside,web]{ieeecolor}
\pdfoutput=1

\usepackage{comment}
\usepackage{generic}
\usepackage{cite}
\usepackage{amsmath,amssymb,amsfonts}
\usepackage{algorithmic}
\usepackage{graphics} 
\usepackage{epsfig} 
\usepackage{times} 
\usepackage{textcomp}
\usepackage{color}
\usepackage{caption}
\usepackage{multirow}
\usepackage{subcaption}
\usepackage{mwe}
\usepackage{enumerate}

\usepackage{amsthm}
\usepackage{mathtools}
\usepackage{bm}
\usepackage{stfloats}

\usepackage[ruled,vlined]{algorithm2e}

\newtheorem{thm}{Theorem}
\newtheorem{cor}{Corollary}
\newtheorem{defn}{Definition}
\newtheorem{prob}{Problem}
\newtheorem{prop}{Proposition}
\newtheorem{assumption}{Assumption}
\DeclareMathOperator*{\argmax}{arg\,max}
\DeclareMathOperator*{\argmin}{arg\,min}

\def\BibTeX{{\rm B\kern-.05em{\sc i\kern-.025em b}\kern-.08em
    T\kern-.1667em\lower.7ex\hbox{E}\kern-.125emX}}
\begin{document}

\title{Multi-agent, Multi-target Path Planning in Markov Decision Processes}
\author{Farhad Nawaz and Melkior Ornik, \textit{Senior Member, IEEE}
\thanks{This work was supported by NASA Early Stage Innovations grant no. 80NSSC19K0209, DARPA grant no. HR001120C0065 and Office of Naval Research grant no. N00014-23-1-2505.}
\thanks{F.~Nawaz was with the University of Illinois Urbana-Champaign, Urbana, IL 61801, USA. He is now with the University of Pennsylvania, Philadelphia, PA 19104, USA (e-mail: farhadn@seas.upenn.edu).}
\thanks{M.~Ornik is with the Department of Aerospace Engineering and the Coordinated Science Laboratory, University of Illinois Urbana-Champaign, Urbana, IL 61801, USA (e-mail: mornik@illinois.edu).}}

\maketitle

\begin{abstract}
Missions for autonomous systems often require agents to visit multiple targets in complex operating conditions. This work considers the problem of visiting a set of targets in minimum time by a team of non-communicating agents in a Markov decision process (MDP). The single-agent problem is at least NP-complete by reducing it to a Hamiltonian path problem. We first discuss an optimal algorithm based on Bellman's optimality equation that is exponential in the number of target states. Then, we trade-off optimality for time complexity by presenting a suboptimal algorithm that is polynomial at each time step. We prove that the proposed algorithm generates optimal policies for certain classes of MDPs. Extending our procedure to the multi-agent case, we propose a target partitioning algorithm that approximately minimizes the expected time to visit the targets. We prove that our algorithm generates optimal partitions for clustered target scenarios. We present the performance of our algorithms on random MDPs and gridworld environments inspired by ocean dynamics. We show that our algorithms are much faster than the optimal procedure and more optimal than the currently available heuristic.
\end{abstract}


\section{Introduction}
\label{sec:introduction}

A common high-level motion planning problem~\cite{oceans, multi_uncertain} is for a team of non-communicating agents to visit multiple target states in minimal time. Complex environmental conditions, e.g., underwater regions~\cite{real_ocean} and planetary bodies~\cite{mars} may make it difficult for agents to communicate with each other in real time. Motivated by the scenario of non-communicating agents, the high-level planning problem is naturally decoupled to a problem of assigning targets to each agent, and a subsequent single agent path planning problem with multiple target states. Namely, we first seek to find a strategy for a single agent to visit a set of target states in minimum time. Then, we consider a task allocation problem of assigning multiple targets to multiple agents such that the joint travel time to visit the targets by a team of agents is minimal.

The complex and partly unknown environment in many path planning scenarios has motivated the use of stochastic dynamics~\cite{stoch_traffic} that model the motion of an agent as a Markov decision process (MDP)~\cite{oceans}. An MDP is a mathematical framework that accounts for the lack of knowledge or complexity of the system dynamics by allowing for stochastic transitions between system states~\cite{puterman}. We interpret the problem of visiting multiple targets by a single agent in an MDP as an extension of a stochastic shortest path (SSP) problem~\cite{bertsekas}. An SSP problem seeks to find a policy that drives the agent operating on an MDP to a target state with minimal expected cost. Classical dynamic programming methods can be adopted for SSP problems~\cite{SSP_opt}. Since our objective is to visit multiple target states instead of a single target state, a modified approach is required. Hence, we propose a strategy that depends on the history of visited target states.


By pursuing a dynamic programming approach exploited in classical SSP problems~\cite{SSP_opt}, we solve our problem for global optimality on a product MDP~\cite{product} that includes the history of visited states in the state space. However, we show that our problem is at least NP-complete because it is a generalization of the Hamiltonian path problem~\cite{tsp_np}. Existing approaches based on approximate dynamic programming are computationally expensive to solve the multi-target problem and are tailored for model-free scenarios~\cite{RL_intro}. One novel contribution of this work is solving a planning problem that is locally optimal at each time step to trade-off optimality and time complexity, while utilizing the known stochastic model of the environment. We also prove that our algorithm is optimal for some classes of MDPs. A similar approach based on value iteration is used in~\cite{demining} to explore a set of targets on a gridworld environment with local sensor information. However, our approach is for a general MDP rather than a gridworld without any local sensor data.

Extending to the multi-agent case, we formulate a task allocation problem~\cite{multi_robot} with the objective of minimizing the time to visit the targets assigned to each agent. The combined target assignment and path planning problem on an MDP is a generalization of the multiple traveling salesman problem~($m$~-~TSP)~\cite{ant_mTSP} with stochastic transition dynamics. The $m$~-~TSP asks to find the shortest possible path for $m$ agents to visit a set of nodes on a complete weighted graph. The objective of task allocation in the $m$-TSP with stochastic transitions is to assign targets to each agent such that the expected time to visit the targets assigned to each agent is minimized. Existing work~\cite{magent_eval} related to $m$-TSP have considered only graphs and minimize the sum of costs between the agents, but we focus on MDPs and minimize the maximum cost between the agents. In our work, we propose a novel partitioning algorithm that builds upon a heuristic presented in~\cite{mTSP} which attempts to solve the task allocation problem in the $m$-TSP. We compute a complete weighted model graph whose nodes are the states of the original MDP and the weights are the optimal expected time to reach every pair of states. We then adopt the algorithm in~\cite{mTSP} for the model graph to generate a partition of the target states that attempts to minimize the time to visit the target states by multiple agents. Each agent then visits the targets assigned to them by following our suboptimal single agent policy to minimize the expected time to visit the target states. Though our partitioning algorithm is suboptimal in general, one contribution of our work is providing conditions of optimality for our heuristic procedure on MDPs where targets are \textit{clustered}. Clusters are defined based on the optimal expected time to reach a state from another state. The presence of clustered target states in an environment is a common scenario in various multi-agent missions~\cite{cluster1, cluster_miss1}.

The rest of the paper is organized as follows. Section~\ref{sec:prel} provides the necessary mathematical background. In Section~\ref{sec:probl}, we formally state the planning problem and the partitioning problem. In Section~\ref{sec:exact_sol_sec}, we discuss the optimal policy for the single agent case. We propose a novel suboptimal single agent planning procedure in Section~\ref{sec:sub_opt_single} and prove that it is optimal for certain classes of MDPs. In Section~\ref{sec:part_multiple}, we propose a partitioning algorithm for the multi-agent case and provide conditions of optimality for our heuristic procedure on clustered target states. Finally, we present the numerical results of our proposed algorithms in Section~\ref{sec:numerical} on random graphs, MDPs, and stochastic gridworld environments. In Section~\ref{sec:numerical}, we also compare the performance of our single agent heuristic to the nearest neighbor algorithm for graphs and conclude the paper with possible future directions in Section~\ref{sec:conclus}.

\section{Preliminaries}
\label{sec:prel}

In this paper, $|S|$ denotes the cardinality of set $S$ and $2^{S}$ indicates the set of all subsets of $S$. Notation $\mathbb{E}^{\pi}[X]$ denotes the expectation of a random variable $X$ when an agent on a Markov decision process (MDP) follows policy $\pi$, whereas $\mathbb{E}[X|Y=y]$ denotes the expectation of $X$ given a value $y$ for another random variable $Y$.

\subsection{Markov Decision Process}
\label{mdp_sec}

An MDP is a mathematical framework to model the motion of agents in a stochastic environment as defined below~\cite{puterman}.

\begin{defn}
A \emph{finite Markov decision process} is a tuple \({\mathcal{M}=(S,A,\mathcal{T})}\), where $S$ is a finite set of states, $A$ is a finite set of actions and ${\mathcal{T}:S\times A\times S\to[0,1]}$ is a transition probability function where ${\sum_{s'\in S}\mathcal{T}(s,a,s')=1 \ \textnormal{for all} \ s\in S,a\in A}$.
\end{defn}
The dynamics of an agent operating in an MDP $\mathcal{M}$ are given as follows. The agent in state $s\in S$ chooses an action $a\in A$ and transitions to a state $s' \in S$ in one time step with probability $\mathcal{T}(s,a,s')$.

A \textit{policy} $\pi_t$ for an agent is defined by the probability $\pi_t(a|s)$ that the agent in state $s$ takes action $a$ at time $t$. A policy $\pi_t$ applied on an MDP $\mathcal{M}$ generates a random process~\cite{multiple_ct} with transition probabilities ${\pi_t(a|s)\mathcal{T}(s,a,s') \ \textnormal{for all} \ a\in A, s\in S}$, at every time $t$. A policy is \textit{deterministic} if at every time $t$ and for all states $s\in S$, there exists some $a\in A$ such that $\pi_t(a|s)=1$. A policy is \textit{stationary} when it is time-independent, i.e., $\pi_t = \pi_{t'}$ for all $t,t'$. A stationary policy applied on an MDP $\mathcal{M}$ induces a \textit{Markov chain}~\cite{multiple_ct}. Every policy generates many possible \textit{paths} on the MDP. A sample path of an agent is a sequence of states $s_0 s_1 ... $, where $\pi_t(a|s_t)\mathcal{T}(s_t,a,s_{t+1})>0$ for all $t\geq0$.
\subsection{Value Function}
\label{vf_sec}

In order to encode the task objective, an MDP is often associated with a reward function $R:S\times A \times S \to \mathbb{R}$~\cite{puterman} and the objective is to maximize the long-term expected reward. A value function quantitatively describes how good a state is to satisfy the objective. The \textit{value function} of a \textit{discounted reward infinite horizon} MDP for a stationary deterministic policy $\pi$~\cite{puterman}  is
\begin{equation}
    V^{\pi}(s) = \mathbb{E}^{\pi}\left[\sum_{t=0}^{\infty}\gamma^t R(s_t,\pi(s_t),s_{t+1}) \bigg|s_0 = s\right],
    \label{value}
\end{equation}
where $\gamma$ is the discount factor and $\pi(s_t)=a_t$. The objective of maximizing the long-term expected reward is the optimization problem $\max_{\pi}V^{\pi}(s) \ \textnormal{for all} \ s\in S$, which returns the \textit{optimal value function} $V^*(s)$. We define an operator $T^{\pi}$~\cite{puterman} for the policy $\pi$ on some arbitrary value function $V(s)$ by
\begin{equation}
        T^{\pi}(V(s)) = \sum_{s' \in S}\mathcal{T}(s,\pi(s),s')\big(R(s,\pi(s),s') + \gamma V(s')\big).
    \label{T_pi}
\end{equation}
A function $V(s)$ is the true value function $V^{\pi}(s)$ for policy $\pi$ if and only if $V(s)=T^{\pi}(V(s))$. Similar to~\eqref{T_pi}, an operator $T$ is given by
\begin{equation}
T(V(s)) = \max_{a \in A}\sum_{s' \in S}\mathcal{T}(s,a,s')\big(R(s,a,s') + \gamma V(s')\big).
    \label{T}
\end{equation}
From Bellman's optimality principle~\cite{bertsekas}, $V(s)$ is the optimal value function $V^{*}(s)$ if and only if $V(s)=T(V(s))$. The operators $T^{\pi}$ and $T$ are monotonic, i.e., if ${V^1(s)\geq V^2(s)}$ for every $s\in S$, then $T^{\pi}(V^1(s))\geq T^{\pi}(V^2(s))$ and ${T(V^1(s))\geq T(V^2(s))}$~\cite{puterman}. The optimal policy $\pi^*$ that maximizes the long-term expected reward can be obtained from the optimal value function $V^*(s)$ as given below~\cite{bertsekas}:
\begin{equation}
\pi^*(s) = \argmax_{a\in A}\sum_{s' \in S}\mathcal{T}(s,a,s')\big(R(s,a,s') + \gamma V^*(s')\big).
\label{policy_optimal}
\end{equation}
As shown in Theorem 17.8 of~\cite{found_ML}, the optimal policy is deterministic for all finite MDPs. Hence, we consider only deterministic policies in this work.

In~\cite{bertsekas}, a cost $c(s,a,s') = -R(s,a,s')$ is associated with an MDP. A stochastic shortest path (SSP) problem~\cite{bertsekas} is a special case of the total cost infinite horizon problem where the agent should reach a goal state with minimum expected cost. In a SSP problem, (i) there is no discounting ($\gamma = 1$), (ii) the target state $s^g\in~S$ is \textit{absorbing}, i.e., $\mathcal{T}(s^g,a,s^g) = 1$ for all $a \in A$, and  (iii) the target state $s^g$ is \textit{cost-free}, i.e., $c(s^g,a,s) = 0$ for all $a \in A, s \in S$, whereas all other transitions incur a positive cost. A stationary deterministic policy $\pi$ is \textit{proper} if, when using the policy $\pi$, there is a positive probability that the agent will eventually reach the target state for all initial states~\cite{bertsekas}.

\subsection{Cover Time}
\label{ct_sec}

In simple terms, the cover time for a discrete-time finite-state random process is the time required by an agent to visit all the states, while the hitting time is the time required to visit a particular state. In this subsection, we formally define \textit{hitting time} and \textit{cover time} for a discrete-time finite-state random process ${X = \{X_t\}, \ \textnormal{where} \ X_t\in S}$ for all $t\geq0$~\cite{multiple_ct}.
\begin{defn}
\textnormal{Hitting time} of a state $s\in S$ is the first time the agent visits $s$, starting from $s_0\in S$. We denote it by the random variable $H_{X, s_0}^s$ and is defined by
\begin{equation}
H_{X, s_0}^s = \min\{t : X_t = s | X_0 = s_0\}.
\label{ht}
\end{equation}
\end{defn}
We denote $H_{\mathcal{M}, s_0}^{s, \pi}$ as the hitting time of state $s$ starting from $s_0$ for the MDP $\mathcal{M}$ with policy $\pi$, and ${\mathbb{E}^{\pi}[H_{\mathcal{M},s_0}^{s}]=\mathbb{E}\left[H_{\mathcal{M}, s_0}^{s, \pi}\right]}$ as the expected hitting time.
\begin{defn}
\textnormal{Cover time} of a set of states $\mathcal{V} \subseteq S$ is the time required by the agent to visit all the states $s \in \mathcal{V}$, starting from $s_0 \in S$. We denote it by the random variable $C_{X, s_0}^{\mathcal{V}}$ and is defined as
\begin{equation}
C_{X, s_0}^{\mathcal{V}} = \min\{t :  H_{X, s_0}^s \leq t \ \textnormal{for all} \ s \in \mathcal{V}\}.
\label{ct}
\end{equation}
\end{defn}
The equation for cover time~\eqref{ct} is equivalent to ${C_{X, s_0}^{\mathcal{V}}=\max_{s\in~\mathcal{V}}H_{X, s_0}^s}$. Similar to hitting time, we denote $C_{\mathcal{M}, s_0}^{\mathcal{V}, \pi}$ as the cover time of $\mathcal{V}$ starting from $s_0$ for the MDP $\mathcal{M}$ with policy $\pi$, and $\mathbb{E}^{\pi}[C_{\mathcal{M}, s_0}^{\mathcal{V}}] = \mathbb{E}\left[C_{\mathcal{M}, s_0}^{\mathcal{V}, \pi}\right]$ as the expected cover time. Since the MDP $\mathcal{M}$ is known in our work, we use the notation $H_{s_0}^{s, \pi}=H_{\mathcal{M}, s_0}^{s, \pi}$ and $C_{s_0}^{\mathcal{V}, \pi}=C_{\mathcal{M}, s_0}^{\mathcal{V}, \pi}$. A Markov chain is \textit{irreducible} if the agent can reach $s'$ starting from $s$ in a finite number of steps with non-zero probability, for every pair of states $s, s' \in S$. If the Markov chain induced by a stationary policy $\pi$ and MDP $\mathcal{M}$ is irreducible, then the expected cover time is finite~\cite{cover_finite}.

Utilizing the above preliminaries, we formulate the problem of optimal cover time in the subsequent section.

\section{Problem Formulation}
\label{sec:probl}

The objective considered in this work is for multiple agents to jointly visit multiple target states in minimal expected time. We make the following assumption on a priori known MDP~$\mathcal{M}$.
\begin{assumption}
There exists some stationary policy $\pi$ such that the resulting Markov chain induced by $\pi$ is irreducible.
\label{ass1}
\end{assumption}
If Assumption~\ref{ass1} does not hold, then there exists some $s\in S$ and $\mathcal{V}\subseteq S$ such that the expected cover time $\mathbb{E}^{\pi}\left[C^{\mathcal{V}}_{s}\right]$ is infinite for all policies~\cite{cover_finite}. 

We consider the following problem statement for a single agent to visit multiple targets.
\begin{prob}
\label{probl_ct}
Let an agent operate in an MDP $\mathcal{M}$. Let ${s_0\in S}$ be its initial state, $\mathcal{V}$ be the set of target states to be covered, and $\mathbb{E}^{\pi}\left[C^{\mathcal{V}}_{s_0}\right]$ be the expected cover time of $\mathcal{V}$ for an agent starting from state $s_0$ and acting under policy~$\pi$. Under Assumption~\ref{ass1}, find a control policy $\pi^*$ that solves the optimization problem
\begin{equation}
\pi^* = \argmin_{\pi}\left(\mathbb{E}^{\pi}\left[C^{\mathcal{V}}_{s_0}\right]\right).
    \label{prob}
\end{equation}
\end{prob}

The expected cover time $\mathbb{E}^{\pi}\left[C^{\mathcal{V}}_{s_0}\right]$ depends both on the initial state $s_0\in S$ and the target set~$\mathcal{V}$. Hence, it is appropriate to consider the motion of the agent on the product state space $S_p=S\times 2^{\mathcal{V}}$ which is often used in the theory of model checking~\cite{product}. The agent's extended state $(s,\overline{\mathcal{V}})\in S_p$ encodes its current location $s \in S$ and the remaining set of states $\overline{\mathcal{V}}\subseteq \mathcal{V}$ to be visited. Hence, a transition in the product space~$S_p$ denotes a change in the state $s\in S$ of the agent and a change in the remaining set of states $\overline{\mathcal{V}}\subseteq\mathcal{V}$ to be visited. The transition probabilities $\mathcal{T}_p:S_p\times A\times S_p\to[0,1]$ are defined as

\begin{multline}
        \mathcal{T}_p\left((s_1, \overline{\mathcal{V}}_1), a, (s_2, \overline{\mathcal{V}}_2)\right) \\ = 
    \begin{cases}
    \mathcal{T}(s_1, a, s_2) & \textnormal{if} \quad \overline{\mathcal{V}}_2 = \overline{\mathcal{V}}_1 \setminus \{s_2\},\\
    0 & \textnormal{otherwise},
    \end{cases} 
    \\ \textnormal{for all} \quad 
    (s_1, \overline{\mathcal{V}}_1), (s_2, \overline{\mathcal{V}}_2) \in S_p, a \in A.
    \label{trans_product}
\end{multline}
The key takeaway from~\eqref{trans_product} is that the agent can never shrink the remaining set of states to be visited by more than one in a single transition. Then, a stationary deterministic policy $\pi$ on the product MDP $\mathcal{M}_p = (S_p, A, \mathcal{T}_p)$ is a mapping from the product space to the action space given by $\pi : S_p \to A$. In our problem, the objective is to visit all states in $\mathcal{V}$. Thus, the target set is $T_p = \{S~\times ~\{\emptyset\}\}$ and we wish to compute a policy that leads the agent to any state $t_p \in T_p$ in minimal expected time.

Using~\eqref{ct} and~\cite{cover_finite}, given a set $\overline{\mathcal{V}} \subseteq S$ to be covered, the expected cover time $\mathbb{E}^{\pi}\left[C_{s}^{\overline{\mathcal{V}}}\right]$ when the agent starts from state $s \in S$ with policy $\pi$ is given by
\begin{multline}
\mathbb{E}^{\pi}\left[C_{s}^{\overline{\mathcal{V}}}\right] = 1 + \\ \sum_{s' \in S} \mathcal{T}_p\left((s,\overline{\mathcal{V}}),\pi(s, \overline{\mathcal{V}}),(s',\overline{\mathcal{V}}\setminus \{s'\})\right)\mathbb{E}^{\pi}\left[C_{s'}^{\overline{\mathcal{V}} \setminus \{s'\}}\right] \\ \textnormal{for all} \ (s, \overline{\mathcal{V}})\in S_p.
    \label{ct_eq}
\end{multline}
Since we consider only states $\left(s', \overline{\mathcal{V}} \setminus \{s'\}\right)$ in~\eqref{ct_eq}, all the zero probabilities mentioned in the second condition of~\eqref{trans_product} are not included in~\eqref{ct_eq}. Based on~\eqref{ct_eq}, we interpret one unit of time as the cost incurred for the immediate one-step state transition and $\mathbb{E}^{\pi}\left[C_{s'}^{\overline{\mathcal{V}} \setminus \{s'\}}\right]$ as the future expected cost for each possible state $s' \in S$ until the agent covers the set $\overline{\mathcal{V}}$. If $\overline{\mathcal{V}} = \emptyset \ \textnormal{or} \ {\overline{\mathcal{V}} = \{s\}}$, then the agent has already covered the required set and cover time is zero:~$
{\mathbb{E}^{\pi}\left[C_{s}^{\{s\}}\right]=\mathbb{E}^{\pi}\left[C_{s}^{\{\emptyset\}}\right]=0} \ \textnormal{for all} \ s\in S$. By Assumption~\ref{ass1}, there exists a policy $\pi' : S \to A$ which generates an irreducible Markov chain. We can define a policy $\pi:S_p \to A$ as $\pi(s,\overline{\mathcal{V}}) = \pi'(s)$ for all $(s,\overline{\mathcal{V}})\in S_p$ so that the expected cover times $\mathbb{E}^{\pi}\left[C_{s}^{\overline{\mathcal{V}}}\right]$ are finite for all $(s,\overline{\mathcal{V}})\in S_p$. 

Since we introduced the MDP problem in Section~\ref{vf_sec} as maximizing the long term expected reward, we multiply~\eqref{ct_eq} by $-1$ to convert cost into reward. We interpret value function~\eqref{T_pi} as negative of the expected cover time. The reward function $R_p:S_p\to\mathbb{R}$ is
\begin{equation}
R_p(s, \overline{\mathcal{V}}) = 
\begin{cases}
0 & \textnormal{if} \ \overline{\mathcal{V}} \in \{\emptyset, \{s\}\}, \\
-1 & \textnormal{otherwise}.
\end{cases}
    \label{reward}
\end{equation}
The reward function described in Section~\ref{vf_sec} for the product space $S_p$ is $R:S_p\times A\times S_p\to\mathbb{R}$, but the reward function in~\eqref{reward} is $R_p:S_p\to\mathbb{R}$. All notions introduced in Section~\ref{sec:prel} can be analogously stated for the reward function in~\eqref{reward} as well. From~\eqref{ct_eq} and~\eqref{reward}, the recursive value function equation for a given policy $\pi$ is
\begin{multline}
V^{\pi}(s,\overline{\mathcal{V}}) = -1 + \\ 
\sum_{s' \in S} \mathcal{T}_p\left((s,\overline{\mathcal{V}}),\pi(s, \overline{\mathcal{V}}),(s',\overline{\mathcal{V}}\setminus \{s'\})\right)V^{\pi}(s', \overline{\mathcal{V}}\setminus \{s'\}) \\ \textnormal{for all} \ (s, \overline{\mathcal{V}})\in S_p,
    \label{vf_eq}
\end{multline}
where $V^{\pi}(s,\overline{\mathcal{V}})=-\mathbb{E}^{\pi}\left[C_{s}^{\overline{\mathcal{V}}}\right]$ is the value function. Therefore, problem~\eqref{prob} produces the same solution as ${\pi^*=\argmax_{\pi}\left(V^{\pi}(s_0, \mathcal{V})\right)}$. Analogously, we can interpret expected cover time~\eqref{ct_eq} as the expected cost and solve the optimization problem~\eqref{prob} to minimize the expected cost. 

Once the agent reaches a target state $t_p~\in~T_p$, from~\eqref{trans_product}, it does not transition anymore, and from~\eqref{reward}, the cost is~0.  Comparing~\eqref{T_pi} and~\eqref{vf_eq}, $\gamma = 1$ in our work, which matches the SSP problem defined in Section~2.1 of~\cite{bertsekas}. Therefore, Problem~\ref{probl_ct} is an extended version of the SSP problem on the product space $S_p$, with a target set $T_p$ instead of one target state. 

Extending the problem from a single agent to a team, we now consider multiple agents operating on the MDP~$\mathcal{M}$, all starting from the same state $s_0 \in S$. In this work, we assume that all the agents follow the same transition dynamics $\mathcal{T}$.   We first define a partition of multiple targets, and then introduce the problem of optimal target assignment to multiple agents before the start of the mission.

\begin{defn}
Let there be $m$ agents operating in an MDP~$\mathcal{M}$ and $\mathcal{V}$ be the set of target states to be covered. Then, $\mathcal{P}~=~\{P_1,P_2,\ldots,P_m\}$ is a partition of the target states~$\mathcal{V}$ for $m$ agents such that $\cup_{i=1}^m P_i = \mathcal{V}, \ P_i \cap P_j = \emptyset$ for all $i,j\in\{1,2,\ldots,m\}, \ i \neq j$.
\label{part_defn}
\end{defn}

For a fixed partition $\mathcal{P}$, the optimal expected time for $m$ agents to jointly visit the target states $\mathcal{V}$ is
\begin{equation}
\min_{\{\pi_1, \pi_2, \ldots, \pi_m\}}\mathbb{E}\left[\max_{P_i} \left(C_{s_0}^{P_i, \pi_i}\right)\right],
    \label{orig_multi_ct}
\end{equation}
where the policy for agent $i\in\{1,2,\ldots,m\}$ is $\pi_i$. Solving the problem in~\eqref{orig_multi_ct} would require us to compute the expectation of the maximum of random variables~$C_{s_0}^{P_i, \pi_i}$ and then jointly optimize over the policies $\{\pi_1, \pi_2, \ldots, \pi_m\}$. Since it is hard to jointly optimize over the expected value of the maximum of random variables~\cite{Emax}, we assume that each agent~$i$ uses the policy $\pi_i^*$ from Problem~\ref{probl_ct} that minimizes the expected cover time for targets $P_i$. Hence, we implicitly make the approximation
\begin{equation}
\min_{\{\pi_1, \pi_2, \ldots, \pi_m\}}\mathbb{E}\left[\max_{P_i} \left(C_{s_0}^{P_i, \pi_i}\right)\right] \approx \max_{P_i} \left(\min_{\pi_i}\mathbb{E}\left[C_{s_0}^{P_i, \pi_i}\right]\right)
    \label{multi_ct_ass}
\end{equation}
by fixing the policy $\pi_i^* = \argmin_{\pi}\mathbb{E}^{\pi}\left[C^{P_i}_{s_0}\right]$ for each agent~$i$ so that we can take the maximum out of the expectation as given in~\eqref{multi_ct_ass} and then solve for the partition~$\mathcal{P}$. Since we interchanged the order of minimizing over the policies and maximizing over the target sets $P_i$,~\eqref{multi_ct_ass} is not an equality in general and establishing the relationship between the two values remains a topic for future work.

The formal problem statement of allocating targets to multiple agents is given below after making the approximation in~\eqref{multi_ct_ass}.
\begin{prob}
\label{probl_multi_ct}
Let there be $m$ agents operating in an MDP $\mathcal{M}$. Let $s_0\in S$ be the initial state of the $m$ agents, and $\mathcal{V}$ be the set of target states to be covered. Let $\mathcal{P}$ be a partition of the target states $\mathcal{V}$ for $m$ agents. Let $\pi_i^*$ and $\mathbb{E}^{\pi_i^*}\left[C_{s_0}^{P_i}\right]$ be the optimal policy and optimal expected cover time for agent $i\in\{1,2,\ldots,m\}$, respectively. Find a partition $\mathcal{P}^*$ that solves the optimization problem
\begin{equation}
\mathcal{P}^* = \argmin_{\mathcal{P}} \max_{i} \left(\mathbb{E}^{\pi_i^*}\left[C^{P_i}_{s_0}\right]\right).
    \label{prob_multi}
\end{equation}
\end{prob}

In subsequent sections, we present the solution methodologies that approximately minimize the expected cover time.

\section{Optimal Policy for Single Agent}
\label{sec:exact_sol_sec}

The objective of Problem~\ref{probl_ct} is to find a policy that results in minimal expected cover time for a single agent. In this section, we adopt \textit{policy iteration}~\cite{RL_intro} to solve for $\pi^*$ of the SSP problem we discussed in Section~\ref{sec:probl}. We describe the optimal policy iteration procedure below, which also serves as the motivation for our heuristic method. 

It is shown in~\cite{SSP_opt} that \textit{policy iteration} --- when starting with a proper policy --- can be adopted for SSP problems. Assumption~\ref{ass1} guarantees the existence of a stationary policy such that the agent can reach $s'$ starting from $s$ in a finite number of steps with non-zero probability, for every pair of states $s, s' \in S$. Hence, we can use Lemma~3 in~\cite{proper_policy} to start with a proper policy that chooses an action uniformly at random for all $(s, \overline{\mathcal{V}}) \in S_p$ where $\overline{\mathcal{V}} \neq {\emptyset}$.

The first step of policy iteration is to recursively compute the value function for the policy $\pi^*(s,\overline{\mathcal{V}})$ using~\eqref{vf_eq}. Since the policy is proper, as shown in Proposition 1.1(a) of~\cite{bertsekas}, recursively computing the value function using~\eqref{vf_eq} converges to the actual value function for policy $\pi$. Then, the new policy is computed using~\eqref{policy_optimal} where $\gamma=1$ for SSP problems. The two steps of value function computation and policy update are repeated until the old policy and new policy are exactly same.

\begin{prop}
Under Assumption~\ref{ass1} and starting from a proper policy, policy iteration on the product MDP $\mathcal{M}_p$ with reward function~\eqref{reward} returns an optimal stationary deterministic policy $\pi^*:S_p \to A$ for Problem~\ref{probl_ct}. 
\label{theorem_SSP}
\end{prop}
\begin{proof}
We showed in Section~\ref{sec:probl} that Problem~\ref{probl_ct} is a SSP problem on the product MDP $\mathcal{M}_p$ with multiple target states. In this section, we described the policy iteration procedure to solve for the optimal policy of a SSP problem. As shown in Proposition~3.5 of~\cite{bertsekas}, the new updated policy is strictly better if the old policy is not optimal. Since the number of proper deterministic policies is finite, policy iteration on the MDP $\mathcal{M}_p$ with reward function~\eqref{reward} always returns the optimal policy to reach a target state with $\overline{\mathcal{V}} \in \{\emptyset, \{s\}\}$.
\end{proof}

Though Proposition~\ref{theorem_SSP} returns the optimal policy, finding it comes at a high computational cost as we will show that Problem~\ref{probl_ct} is NP-complete by reducing it to a Hamiltonian path problem~\cite{tsp_np}. A graph is a deterministic MDP, where the transition probability function satisfies ${\mathcal{T}(s,a,s')\in\{0,1\} \ \textnormal{for all} \ s,s'\in S, \ a\in A}$. A Hamiltonian path~\cite{tsp_np} visits all the states of a graph exactly once. If a Hamiltonian path exists for a graph and the set of targets is the state space $S$, then Problem~\ref{probl_ct} seeks to find a Hamiltonian path on the graph. However, the problem of determining whether a Hamiltonian path exists for a graph is NP-complete~\cite{tsp_np} in the number of states. Hence, Problem~\ref{probl_ct} is at least as hard as minimizing the cover time on a graph. Since the value function is computed for all $(s, \overline{\mathcal{V}}) \in S_p$ in the policy iteration procedure, it consumes $O\left(|S|^2\left(2^{|\mathcal{V}|}\right)\right)$ number of operations which is exponential time. Since we discussed using the Hamiltonian path problem that Problem~\ref{probl_ct} is at least NP-complete, the policy iteration procedure indeed achieves exponential time complexity for Problem~\ref{probl_ct}.

In the next section, we propose a method with polynomial time complexity at each time step.

\section{Suboptimal Value Iteration for Single agent}
\label{sec:sub_opt_single}

In this section, we present a heuristic method for Problem~\ref{probl_ct} by adopting an approximate value function. We prove that our procedure optimally reaches an unvisited target state at every time step and also prove the optimality of our algorithm on a small class of deterministic MDPs.

The intuition behind our procedure is to, at every time $t$, compute a policy that drives the agent to an unvisited state in $\mathcal{V}$ in minimum expected time. We use the computational procedure in policy iteration as a motivation to compute an approximate value function based on value iteration. Let, at time $t$, the remaining set of states to be visited be $\overline{\mathcal{V}}_t$. We propose a reward function $\widehat{R}_t:S\to\mathbb{R}$ by
\begin{equation}
\widehat{R}_t(s) = 
\begin{cases}
-|\overline{\mathcal{V}}_t| & \textnormal{if} \ s \notin \overline{\mathcal{V}}_t, \\
-|\overline{\mathcal{V}}_t|+1 & \textnormal{if} \ s \in \overline{\mathcal{V}}_t.
\end{cases}
\label{reward_new}
\end{equation}
At every time $t$, we aim to find a policy that yields maximum long-term expected reward using~\eqref{reward_new} for the MDP $\mathcal{M}$. The optimal value function $\widehat{V}^*_t(s)$ for such a problem is
\begin{equation}
\widehat{V}^*_t(s) = 
\max_{\widehat{\pi}_t}\mathbb{E}^{\widehat{\pi}_t}\left[\sum_{k=0}^{\infty}\gamma^k \widehat{R}_t(s_k) \bigg|s_0 = s\right],
    \label{value_hat}
\end{equation}
where $\widehat{\pi}_t : S \to A$ is the deterministic stationary policy generated at time~$t$. Let the optimal policy which maximizes~\eqref{value_hat} be $\widehat{\pi}^*_t$.
Then, using Bellman's optimality principle~\cite{bertsekas}, $\widehat{V}^*_t(s)$ and $\widehat{\pi}_t^*(s)$ can be computed as
\begin{multline}
     \widehat{V}^*_t(s) = \max_{a \in A} \sum_{s' \in S}\mathcal{T}(s,a,s')\Big(\widehat{R}_t(s') + \gamma \widehat{V}^*_t(s')\Big), \\ \textnormal{for all} \ s \in S,
    \label{value_new}
\end{multline}
\begin{multline}
    \widehat{\pi}_t^*(s) = \argmax_{a \in A} \sum_{s' \in S}\mathcal{T}(s,a,s')\Big(\widehat{R}_t(s') +  \gamma \widehat{V}^*_t(s')\Big), \\ \textnormal{for all} \ s \in S.
    \label{policy_new}
\end{multline}
The discount factor is $\gamma \in [0,1)$. Note that we used a discount factor of $\gamma=1$ in Section~\ref{sec:exact_sol_sec} to solve for the optimal policy since it was a SSP problem on the product MDP~$\mathcal{M}_p$. We can't use a discount factor of $\gamma=1$ in this section because the infinite sum in~\eqref{value_hat} will have a finite value only if $\gamma\in[0,1)$. Let at time $t$ the state of the agent be $s_t\in S$. If $\gamma=0$,~\eqref{policy_new} returns an action that transitions the agent to a state in $\overline{\mathcal{V}}_t$ in exactly one transition with maximal probability. If there is no such one-step transition, i.e., $\mathcal{T}(s_t,a,s') = 0$ for all ${s'\in\overline{\mathcal{V}}_t, a\in A}$ when $\gamma=0$, then $\widehat{\pi}^*_t(s_t)$ has multiple solutions. We use the convention that if $\widehat{\pi}^*_t(s_t)$ is not unique, then the agent chooses one of the actions selected at random since all such actions return the optimal value function $\widehat{V}^*_t(s)$ for $\gamma=0$. Thus, if $\gamma=0$ and the MDP is a graph, then~\eqref{policy_new} is exactly the \textit{nearest neighbor} heuristic~\cite{heur_TSP}. However, our heuristic presented in this section aims to find a path of minimal length to an unvisited state. As we increase~$\gamma$, the emphasis on rewards of future time steps also increases.

From~\eqref{value_hat} and~\eqref{reward_new}, we note that the approximate value function $\widehat{V}_t^*(s)$ changes if and only if $\overline{\mathcal{V}}_t$ changes. Therefore, the benefit of~\eqref{value_new} is that the value function depends only on the states $s \in S$ and we recompute $\widehat{V}_t^*(s)$  only when the agent visits an unvisited state $s \in \overline{\mathcal{V}}_t$. Unlike the policy iteration procedure described in Section~\ref{sec:exact_sol_sec}, where we need to compute the value function for all states $S$ and all subsets of $\mathcal{V}$, our new proposed procedure computes the value function only for all states $S$. 

We present our procedure in Algorithm~\ref{alg_heur_ct}. The value function $\widehat{V}_t(s)$ is computed for all states in lines 6--12 using value iteration and~\eqref{value_new}. The policy is computed in line~14 as given in~\eqref{policy_new}. Lines 14--17 are repeated until the agent visits a state in $\overline{\mathcal{V}}_t$. Once the agent visits a state in $\overline{\mathcal{V}}_t$, the set of states to be covered is updated and denoted by $\overline{\mathcal{V}}_{t+1}$. Then, the new value function is computed for $\overline{\mathcal{V}}_{t+1}$. The iterative procedure repeats until the agent has visited all the states in $\mathcal{V}$.

\LinesNumbered
\begin{algorithm}[!t]
\caption{Suboptimal Value Iteration}
\label{alg_heur_ct}
\SetAlgoLined
\textbf{Input}: $S, A, \mathcal{T}, \mathcal{V}, \varepsilon, s_0, \gamma$\;
\textbf{Initialize:} $t=0, \overline{\mathcal{V}}_0 = \mathcal{V} \setminus \{s_0\}$\;
$\widehat{V}_t(s) = - \infty \ \textnormal{for all} \ s \in S$\;
\While{$\overline{\mathcal{V}}_t \neq \emptyset$}{
$\Delta = \infty$\;
\While{$\Delta \geq \varepsilon$}{
  \ForAll{$s \in S$}
    {
    $\widehat{V}_t'(s) \leftarrow \widehat{V}_t(s)$\;
     $\widehat{V}_t(s) \leftarrow \max_{a \in A}\sum_{s' \in S}\mathcal{T}(s,a,s')\Big(\widehat{R}_t(s') + \gamma \widehat{V}_t(s')\Big)$\;
     }
  $\Delta \leftarrow \max_{s}|\widehat{V}_t(s) - \widehat{V}'_t(s)|$
}
\Repeat{$|\overline{\mathcal{V}}_t| = |\overline{\mathcal{V}}_{t-1}|-1$}{
$\widehat{\pi}^*_t(s_t)\leftarrow \argmax_{a \in A} \sum_{s' \in S}\mathcal{T}(s_t,a,s')\Big(\widehat{R}_t(s') +  \gamma \widehat{V}_t(s')\Big)$\;
Implement action $\widehat{\pi}^*_t(s_t)$\;
$t=t+1$\;
Update set to be covered: $\overline{\mathcal{V}}_t = \overline{\mathcal{V}}_{t-1} \setminus \{s_t\}$\;
}
}
 \end{algorithm}

Now, we seek to prove that Algorithm~\ref{alg_heur_ct} will terminate in finite time. We use the following claim.

\begin{prop}
Under Assumption~\ref{ass1}, there exists some $\overline{\varepsilon} > 0$ such that for all $\varepsilon\in(0, \overline{\varepsilon})$ and any $\gamma \in (0,1)$, at every time $t\geq0$, Algorithm~\ref{alg_heur_ct} generates a policy that drives the agent starting from state $s_t \in S$ to eventually reach some state $s \in \overline{\mathcal{V}}_t$ with probability 1, and reach it with minimal expected time. 
\label{prop_1}        
\end{prop}
\begin{proof}
In Algorithm~\ref{alg_heur_ct}, we use a stopping criterion $\varepsilon$ which returns an estimate $\widehat{V}_t(s)$ of the optimal value function $\widehat{V}_t^*(s)$ that solves $T(\widehat{V}_t^*(s)) = \widehat{V}_t^*(s)$. The value function computed based on value iteration in lines~7--11 uses the optimal Bellman operator $T$ as described in Section~\ref{vf_sec}. Since the operator $T$ is monotone~\cite{puterman}, the value function computed using value iteration converges to the optimal value function. Thus, as shown on page~27 of~\cite{bertsekas}, there exists a small enough~$\overline{\varepsilon}>0$ such that for all $\epsilon\in(0,\overline{\varepsilon})$, the policy computed in line~14 is optimal after a finite number of iterations of value iteration present in lines~7-11 of Algorithm~1. The computed policy is  optimal with respect to the reward function defined in~\eqref{reward_new}. From~\eqref{reward_new}, at every time $t$, all states $s\in\overline{\mathcal{V}}_t$ have the same reward $-|\overline{\mathcal{V}}_t| + 1$, which is higher than the reward $-|\overline{\mathcal{V}}_t|$ assigned for the states $s\in S\setminus\overline{\mathcal{V}}_t$. Hence, the optimal policy computed in line~14 will eventually drive the agent to some state $s\in\overline{\mathcal{V}}_t$ with probability 1, in minimal expected time.
\end{proof}

\begin{cor}
Under Assumption~\ref{ass1} and any $\gamma \in (0,1)$, Algorithm~\ref{alg_heur_ct} will eventually cover the given set of states $\mathcal{V}$ with probability 1.
\label{alg_heur_ct_th}
\end{cor}
\begin{proof}
As shown in Proposition~\ref{prop_1}, the agent reaches a state in $\overline{\mathcal{V}}_t$ in minimal expected time, i.e., $|\overline{\mathcal{V}}_{t'}|=|\overline{\mathcal{V}}_{t}| - 1$ for some finite time $t'>t$, eventually leading to $|\overline{\mathcal{V}}_t|=0$. 
\end{proof}

The value function update in line~9 of Algorithm~\ref{alg_heur_ct} has time complexity $O(|S||A|)$. Lines~7-11 iterates over all possible states and hence it requires $O(|S|^2|A|)$ operations. Then, the time complexity to compute the policy in line~14 is $O(|S||A|)$. Therefore, in a single time step, Algorithm~\ref{alg_heur_ct} consumes only $O(|S|^2|A| + |S||A|)$ operations, in comparison to $O\left(|S|^2\left(2^{|\mathcal{V}|}\right)\right)$ operations consumed by the optimal policy iteration procedure described in Section~\ref{sec:exact_sol_sec}.

The inherent suboptimality of Algorithm~\ref{alg_heur_ct} is a consequence of the value function $\widehat{V}_t(s)$ which does not always capture the complete information about the minimum expected cover time. However, as we will show in Section~\ref{sec:numerical} using numerical results that $\widehat{V}_t(s)$ often produces a reasonable approximation of the minimal expected cover time. Additionally, we can prove that the policy from~\eqref{policy_new} is optimal for path graphs, cycle graphs, and complete graphs. These undirected graphs are illustrated in Fig.~\ref{fig:opt_graphs}; we define them as follows. If the states (vertices) are denoted by $S=\{1, 2, \ldots, n\}$ and the \textit{terminal vertices} are $1$ and $n$, the \textit{path graph} has edges $E_p=\{\{i,i+1\}~|~i=1,2,\ldots,n-1\}$, and the \textit{cycle graph} has edges $E_p \cup \{\{1,n\}\}$, whereas the \textit{complete graph} has an edge between every pair of vertices.

\begin{figure}[!t]
     \centering
     \begin{subfigure}[h]{\linewidth}
         \centering
         \includegraphics[width=0.5\linewidth]{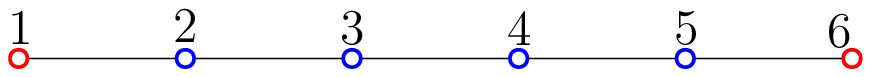}
         \caption{Path graph}
         \label{fig:path}
     \end{subfigure}\\
     \begin{subfigure}[h]{0.4\linewidth}
         \centering
         \includegraphics[width=0.5\linewidth]{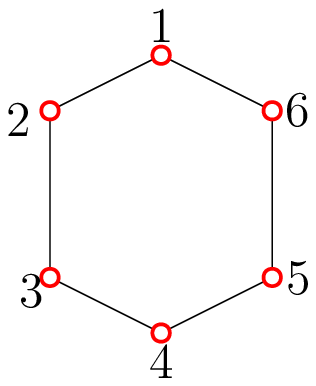}
         \caption{Cycle graph}
         \label{fig:cycle}
     \end{subfigure}
     \begin{subfigure}[h]{0.4\linewidth}
         \centering
         \includegraphics[width=0.5\linewidth]{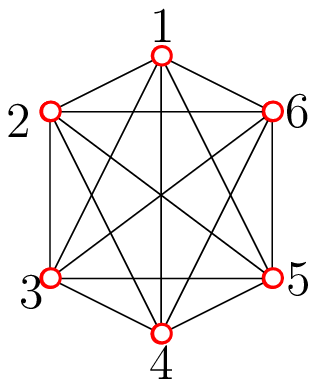}
         \caption{Complete graph}
         \label{fig:complete}
     \end{subfigure}
     \caption{An illustration of the graphs where Theorem~\ref{theorem_dir_heur} applies, with possible initial states in red.}
        \label{fig:opt_graphs}
\end{figure}

\begin{thm}
Algorithm~\ref{alg_heur_ct} always produces optimal paths with any $\gamma \in (0,1)$ for Problem \ref{probl_ct} to cover any given set of targets for the following deterministic MDPs: (i) path graphs with initial state as either of the terminal vertices, (ii) cycle graphs, and (iii) complete graphs.
\label{theorem_dir_heur}
\end{thm}
\begin{proof}
The result follows directly from Proposition~\ref{prop_1} by noting that Algorithm~\ref{alg_heur_ct} updates the target set each time the agent visits a target. For complete graphs, since the agent can reach any state from any other state in exactly one step, Algorithm 1 trivially visits the target states optimally, without visiting any state that is not a target. In path graphs with terminal initial conditions, there is only one direction the agent can start to move and hence that path is optimal. In cycle graphs, the agent has two options for the
initial direction of motion. From Proposition 2, the agent chooses the direction which takes it to the nearest
unvisited target. Once that target is visited, the rest of the problem is like a path graph.
\end{proof} 
In the next section, we consider the task of jointly visiting multiple targets by a team of agents. We discuss the optimal solution to the problem of partitioning targets to multiple agents as defined in Problem~\ref{probl_multi_ct} and propose a suboptimal procedure. 

\section{Partitioning Targets to Multiple Agents}
\label{sec:part_multiple}

Multi-agent planning in our framework has two stages. The first stage partitions the targets to agents as defined in Problem~\ref{probl_multi_ct}, and the second stage consists of optimal planning for each agent described in Problem~\ref{probl_ct}. In Section~\ref{sec:exact_sol_sec}, we showed that optimal planning for single agent is NP-complete. However, even if Problem~\ref{probl_ct} is solved to optimality for all agents operating on a deterministic MDP, Problem~\ref{probl_multi_ct} reduces to a graph partitioning problem, which is provably NP-hard~\cite{graph_part}.

Previous work~\cite{mTSP} which solves a partitioning problem similar to Problem~\ref{probl_multi_ct} deals with a multi-agent traveling salesman problem on a complete graph. Though our work also aims to solve a partitioning problem, it deals with a Markov decision process and not necessarily a complete graph. Hence, we base our approach to Problem~\ref{probl_multi_ct} on ideas from~\cite{mTSP}, but substantially adapting their method for a stochastic environment. We also present some theoretical conditions under which our algorithm generates optimal partitions on an environment with clustered targets.

\subsection{Average length of a Hamiltonian path}
\label{avg_ham_sec}

Since the traveling salesman problem is defined on a weighted complete graph, we compute a weighted complete directed graph $\mathcal{G} = (S, E)$ for the MDP $\mathcal{M}$ such that the weights are ${w(s_1, s_2)=\min_{\pi}\mathbb{E}^{\pi}\left[C^{\{s_2\}}_{s_1}\right]=\min_{\pi}\mathbb{E}^{\pi}\left[H^{s_2}_{s_1}\right]}$ for all ${s_1,s_2\in S}$. The weights $w(s_1, s_2)$ are the optimal expected hitting times for every pair of states $s_1, s_2 \in S$ and all the weights are finite because of Assumption~\ref{ass1}. The weights can be computed using the policy iteration procedure described in Section~\ref{sec:exact_sol_sec} with $\mathcal{V}=\{s_2\}$ and $s_0 = s_1$ for all $s_1, s_2\in~S$ which requires $O(|S|^3|A|)$ number of operations. Hence, the time to compute graph $\mathcal{G}$ is polynomial in the number of states and actions. 

As discussed in Section~\ref{sec:exact_sol_sec}, since computing the optimal cover time is NP-complete, we adopt the idea used in~\cite{mTSP} to approximate the optimal cover time. Let $\mathcal{G}_i$ be the sub-graph of $\mathcal{G}$ induced by $P_i$, where $P_i$ are the target states assigned to agent $i$ as defined in Problem~\ref{probl_multi_ct}. We denote the average length of a Hamiltonian path which covers all the states in $\mathcal{G}_i$ starting from $s_0 \in S$ by $L_a(\mathcal{G}_i)$. We use the average length of a Hamiltonian path as a heuristic for the optimal cover time $\mathbb{E}^{\pi_i^*}\left[C^{P_i}_{s_0}\right]$ for all $i\in~\{1,2,\ldots,m\}$. Since we consider ${s_0\in S}$ as the initial state throughout this work, we don't explicitly mention $s_0$ in $L_a(\mathcal{G}_i)$. Let $|P_i| = n_i$ for all $i\in~\{1,2,\ldots,m\}$. Since there are $n_i(n_i-1)$ edges on the sub-graph induced by the states $P_i$, and there are $n_i-1$ edges on a path to cover all the states in $P_i$, 
\begin{equation}
\begin{split}
L_a(\mathcal{G}_i) = (n_i-1)\left(\frac{\sum_{s_1 \in P_i}\sum_{s_2 \in P_i}w(s_1, s_2)}{n_i(n_i-1)}\right) + \\ \frac{\sum_{s \in P_i}w(s_0, s)}{n_i}, \\
\Rightarrow L_a(\mathcal{G}_i) = \frac{\sum_{s_1 \in P_i}\sum_{s_2 \in P_i}w(s_1, s_2) + \sum_{s\in P_i}w(s_0, s)}{n_i}.
    \label{avg_ham}
    \end{split}
\end{equation}

Using~\eqref{avg_ham}, we can compute $L_a(\mathcal{G}_i)$ for each ${i\in\{1,2,\ldots,m\}}$ in $O(n_i^2)$. Since~${\sum_{i=1}^mn_i = |\mathcal{V}|}$ and ${\sum_{i=1}^mn_i^2 < |\mathcal{V}|^2}$, we can compute $L_a(\mathcal{G}_i)$ for all $i\in~\{1,2,\ldots,m\}$ in $O(|\mathcal{V}|^2)$. Hence, instead of attempting to solve~\eqref{prob_multi} in Problem~\ref{probl_multi_ct}, we try to find a partition $\mathcal{P}'$ that solves the below optimization problem using $L_a(\mathcal{G}_i)$ from~\eqref{avg_ham}:
\begin{equation}
\mathcal{P}' = \argmin_{\mathcal{P}} \max_{i} \left(L_a(\mathcal{G}_i)\right).
    \label{prob_multi_avg}
\end{equation}
The problem in~\eqref{prob_multi_avg} is the same problem in~\eqref{prob_multi}, but replaces the optimal cover time $\mathbb{E}^{\pi_i^*}\left[C^{P_i}_{s_0}\right]$ with the average length of the Hamiltonian path $L_a(\mathcal{G}_i)$ on the sub-graph $\mathcal{G}_i$. 

The optimization problem in~\eqref{prob_multi_avg} is NP-hard because it reduces to the popular number partition problem (NPP)~\cite{NPP}, which is NP-hard as shown in~\cite{mTSP}. The NPP seeks to divide a set $\mathcal{S}$ of positive integers with a fixed even sum $K$ into two subsets $\{\mathcal{S}_1, \mathcal{S}_2\}$ such that the sum of numbers in each subset is exactly $\frac{K}{2}$. Similarly, given a complete graph $\mathcal{G}$ and ${\mathcal{V}\subseteq S\setminus~\{s_0\}}$,~\eqref{prob_multi_avg} asks us to find a partition $\{\mathcal{G}_1, \mathcal{G}_2\}$ where $\mathcal{G}_1\cup \mathcal{G}_2 = \mathcal{V}$ such that $\max \{L_a(\mathcal{G}_1), L_a(\mathcal{G}_2)\} = \frac{K}{2}$ for $m=2$ agents. Thus, Problem~\eqref{prob_multi_avg} is NP-hard and we adopt the heuristic presented in~\cite{mTSP} for the graph $\mathcal{G}$.

\subsection{Partitioning by transfers and swaps}
\label{part_trans_swaps}

The heuristic partitioning procedure searches for a series of transfers and swaps of states between pairs of sub-graphs $(\mathcal{G}_i, \mathcal{G}_k)$ to decrease $\max\{L_a(\mathcal{G}_i), L_a(\mathcal{G}_k)\}$. We denote the sum of edge weights of a sub-graph $\mathcal{G}_i$ by

\begin{equation}
W(\mathcal{G}_i) = \sum_{s_1 \in P_i}\sum_{s_2 \in P_i}w(s_1, s_2) + \sum_{s\in P_i}w(s_0, s),
    \label{sum_edge}
\end{equation}
and the contribution of a state $s\in S$ to $W(\mathcal{G}_i)$ by
\begin{equation}
\Delta W(\mathcal{G}_i, s) = \sum_{s_1 \in P_i}w(s_1, s) + \sum_{s_1 \in P_i}w(s, s_1) + w(s_0, s).
    \label{sum_edge_contr}
\end{equation}
From~\eqref{sum_edge} and~\eqref{avg_ham}, we can write the average length of Hamiltonian path on a sub-graph $\mathcal{G}_i$ by
\begin{equation}
L_a(\mathcal{G}_i) = \frac{W(\mathcal{G}_i)}{n_i},
    \label{size_sub_graph}
\end{equation}
and the maximum average length of Hamiltonian path for a partition $\mathcal{P}$ by
\begin{equation}
M_a(\mathcal{P}) = \max_i L_a(\mathcal{G}_i).
    \label{size_part}
\end{equation}
Let us analyze the transfer of a state $s_i\in P_i$ from $\mathcal{G}_i$ to $\mathcal{G}_k$. After the transfer, for the sub-graphs $\mathcal{G}_i'$ and $\mathcal{G}_k'$,
\begin{equation}
\begin{split}
    L_a(\mathcal{G}_i') = \frac{W(\mathcal{G}_i) - \Delta W(\mathcal{G}_i, s_i)}{n_i-1}, \\
    L_a(\mathcal{G}_k') = \frac{W(\mathcal{G}_k) + \Delta W(\mathcal{G}_k, s_i)}{n_k+1}.
\end{split}
    \label{trans_G_i}
\end{equation}
Transfers from $\mathcal{G}_i$ to $\mathcal{G}_k$ are more useful when $L_a(\mathcal{G}_i)\gg~L_a(\mathcal{G}_k)$. A transfer might still be beneficial if $\max\{L_a(\mathcal{G}_i), L_a(\mathcal{G}_k)\}$ is reduced. However, when $L_a(\mathcal{G}_i)$ and $L_a(\mathcal{G}_k)$ are nearly equal, then a transfer might not reduce $\max\{L_a(\mathcal{G}_i), L_a(\mathcal{G}_k)\}$. In this case, swapping two states between the sub-graphs might reduce $\max\{L_a(\mathcal{G}_i), L_a(\mathcal{G}_k)\}$. Let $s_i\in P_i$ be moved from $\mathcal{G}_i$ to $\mathcal{G}_k$, and $s_k\in~P_k$ be moved from $\mathcal{G}_k$ to $\mathcal{G}_i$ simultaneously. Then, the sum of the edge weights of the new sub-graphs $\mathcal{G}_i'$ and $\mathcal{G}_k'$ are
\begin{equation}
\begin{split}
    W(\mathcal{G}_i') = W(\mathcal{G}_i) - \Delta W(\mathcal{G}_i, s_i) + \Delta W(\mathcal{G}_i, s_k) - \\ w(s_i, s_k) - w(s_k ,s_i), \\
    W(\mathcal{G}_k') = W(\mathcal{G}_k) - \Delta W(\mathcal{G}_k, s_k) + \Delta W(\mathcal{G}_k, s_i) - \\ w(s_i, s_k) - w(s_k ,s_i).
\end{split}
    \label{swap_G_i}
\end{equation}

From~\eqref{swap_G_i}, $L_a(\mathcal{G}_i')$ and $L_a(\mathcal{G}_k')$ can be computed using~\eqref{size_sub_graph}. Since there are $n_i$ potential transfers from $\mathcal{G}_i$, and $n_in_k$ potential swaps between $\mathcal{G}_i$ and $\mathcal{G}_k$, the best transfer and swap can be computed in $O(|\mathcal{V}|^2)$.

We present our heuristic partitioning algorithm in Algorithm~\ref{alg_part}. In practice, we use a heuristic procedure~\cite{greedy_k} to generate the initial partition $\mathcal{P}$ which will be described in Section~\ref{sec:numerical} when presenting the numerical results. A substantial difference of our procedure from the algorithm presented in~\cite{mTSP} is the computation of model graph~$\mathcal{G}(S, E)$. Also,~\cite{mTSP} does either a swap or a transfer at every iteration, but Algorithm~\ref{alg_part} could possibly do both swaps and transfers at each iteration. Since there are only finite number of swaps and transfers possible, and we recheck a partition $\mathcal{P}$ only if $M_a(\mathcal{P})$ has strictly reduced in the previous iteration, Algorithm~\ref{alg_part} terminates in finite time, generating a local optimum for the problem in~\eqref{prob_multi_avg}.

We will prove that Algorithm~\ref{alg_part} will generate optimal partitions $\mathcal{P}^*$ in Problem~\ref{probl_multi_ct} for MDPs with \textit{clustered} target states as defined below. The clusters are defined based on the hitting time between states. We denote $w_{c}$ to be the \textit{worst case hitting time} for pairs of target states within the same cluster, and $w_{l}$ to be the \textit{best case hitting time} for pairs of target states in different clusters.

\begin{defn}
Given an MDP $\mathcal{M}(S,A,\mathcal{T})$, we define a set of target states $\mathcal{V}$ to be \emph{clustered} if $\mathcal{V}$ can be partitioned to $m$ clusters as $\mathcal{R}(w_{c}, w_{l}, \mathcal{V})=\{R_1, R_2, \ldots, R_m\}$ where each cluster $R_i = \{s_i^1, s_i^2, \ldots, s_i^n\}$ has $n$ targets, ${\cup_{i=1}^m R_i=\mathcal{V}\subseteq S}$, and ${R_i\cap R_j=\emptyset}$ for all ${i, j\in\{1,2,\ldots,m\}, i\neq j}$. Additionally,
\begin{equation}
\begin{split}
H_{s_1}^{s_2, \pi^*_{s_1, s_2}} \in [w_c', w_c] \ \textnormal{for all} \ s_1, s_2 \in R_i, s_1 \neq s_2, \\ i \in \{1,2,\ldots,m\},
    \label{in_cluster_bounds_ht}
    \end{split}
    \end{equation}
\begin{equation}
\begin{split}
H_{s_1}^{s_2, \pi^*_{s_1, s_2}} \in [w_{l}, w_l'] \ \textnormal{for all} \ s_1 \in R_i, s_2 \in R_j, \\ i,j \in \{1,2,\ldots,m\}, i \neq j,
    \label{out_cluster_bounds_ht}
    \end{split}
\end{equation}
where policy $\pi^*_{s_1, s_2} : S \times S \to A$ is the optimal policy that solves Problem~\ref{probl_ct} with $\mathcal{V}=\{s_2\}$ and $s_0 = s_1$, and $w_c \leq w_l$. We also define $w_1$ and $w_2$ so that
\begin{equation}
H_{s_0}^{s, \pi^*_{s_0, s}} \in [w_{1}, w_{2}] \ \textnormal{for all} \ s\in~\mathcal{V}.
\label{ht_initial_bound}
\end{equation}

\label{cluster_defn}
\end{defn}

\LinesNumbered
\begin{algorithm}[!t]
\caption{Partitioning by transfers and swaps}
\label{alg_part}
\SetAlgoLined
\textbf{Input}: $S, A, P, \mathcal{V}, m, s_0$\;
Compute $\mathcal{G}(S, E)$ using policy iteration\;
\textbf{Initialize:} A partition $\mathcal{P}$; all subgraphs are unchecked for swap and for transfer\;
\Repeat{$M_a(\mathcal{P})$ \ \textnormal{has not reduced}}{
\ForAll{\textnormal{pairs of unchecked subgraphs} $(\mathcal{G}_i, \mathcal{G}_k)$}{
$S_{min} = \max\left(L_a(\mathcal{G}_i), L_a(\mathcal{G}_k)\right)$\;
Compute $L_a(\mathcal{G}_i')$ and $L_a(\mathcal{G}_k')$ using~\eqref{swap_G_i}\;
\eIf{$\max\left(L_a(\mathcal{G}_i'), L_a(\mathcal{G}_k')\right)<~S_{min}$}{
Swap the states $(s_i, s_k)$ between $(\mathcal{G}_i, \mathcal{G}_k)$\;
}{
$(i, k) \leftarrow $ checked for swap\;
}
Compute $L_a(\mathcal{G}_i')$ and $L_a(\mathcal{G}_k')$ using~\eqref{trans_G_i}\;
\eIf{$\max\left(L_a(\mathcal{G}_i'), L_a(\mathcal{G}_k')\right)<~S_{min}$}{
Transfer the state $s_i$ from $\mathcal{G}_i$ to $\mathcal{G}_k$\;
}{
$(i, k) \leftarrow $ checked for transfer\;
}
Update the partition $\mathcal{P}$\;
}
}
\textnormal{\textbf{Output:} Partition $\mathcal{P}$}
\end{algorithm}

On graph $\mathcal{G}(S, E)$, since the weights are defined as ${w(s_1, s_2)=\min_{\pi}\mathbb{E}\left[H_{s_1}^{s_2, \pi}\right]=\mathbb{E}\left[H_{s_1}^{s_2, \pi^*_{s_1, s_2}}\right]}$ for all ${s_1, s_2\in S}$, the bounds in~\eqref{in_cluster_bounds_ht},~\eqref{out_cluster_bounds_ht} and~\eqref{ht_initial_bound} also apply to the weights $w(s_1, s_2)$. From~\eqref{in_cluster_bounds_ht} and~\eqref{out_cluster_bounds_ht}, the bound for weights of states within the same cluster is
\begin{equation}
\begin{split}
w(s_1, s_2) \in [w_c',w_c] \ \textnormal{for all} \ s_1, s_2 \in R_i, s_1 \neq s_2, \\ i \in \{1,2,\ldots,m\},
    \label{cluster_bound_alg}
    \end{split}
    \end{equation}
and the bound for weights of states in different clusters is
\begin{equation}
\begin{split}
w(s_1, s_2) \in [w_{l}, w_l'], \ \textnormal{for all} \ s_1 \in R_i, s_2 \in R_j, \\ i,j \in \{1,2,\ldots,m\}, i \neq j.
    \label{out_cluster_bounds_alg}
    \end{split}
\end{equation}

In Fig.~\ref{fig:theorem}, we depict a deterministic MDP (graph) with ${m=3}$ clusters and $n=4$ targets in each cluster which satisfies the bounds in Definition~\ref{cluster_defn} where ${w_{c}=2, w_{l}=26, w_1=13, w_2=15}$. Such large difference between $w_c$ and $w_l$ will be more clear from the below theorem. Having defined an MDP with clustered target states, we expect the partition $\mathcal{R}$ to be optimal for Problem~\ref{probl_multi_ct} when $w_l$ is sufficiently larger than $w_c$. Indeed, we prove this claim in the following theorem.

\begin{thm}
Under Assumption~\ref{ass1}, let the set of target states $\mathcal{V}$ on an MDP {$\mathcal{M}(S, A, \mathcal{T})$} be clustered as given in Definition~\ref{cluster_defn} and
\begin{equation}
    \begin{split}
    w_{l} > (n-1)w_{c} + (w_2 - w_1).
        \label{cluster_opt_result}
    \end{split}
\end{equation}
Then, the clustered partition $\mathcal{R}$ is the optimal partition $\mathcal{P}^*$ for Problem~\ref{probl_multi_ct}.
\label{cluster_is_optimal}
\end{thm}

\begin{figure}[!t]
         \centering        \includegraphics[width=0.7\linewidth]{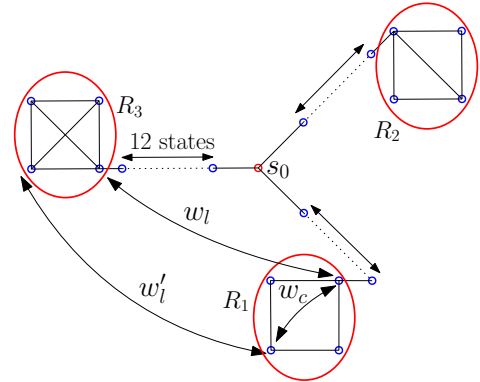}
         \caption{A clustered graph which satisfies the bounds in Definition~\ref{cluster_defn}.}
         \label{fig:theorem}
\end{figure}

\begin{proof}
We first consider the hitting time between states in a partition as given in Definition~\ref{cluster_defn}, but a pair of states are in the ``wrong clusters". Let that partition be ${\mathcal{R}'=\{R_1, R_2, \ldots, R_i', R_j', \ldots, R_m\}}$ where ${R_i'=\{s_{i}^1, s_{i}^2, \ldots, s_{j}^k, \ldots, s_{i}^n\}}$ and ${R_j'=\{s_{j}^1, s_{j}^2, \ldots, s_{i}^k, \ldots, s_{j}^n\}}$. The state $s_i^j$ denotes a state that belongs to cluster $i$ as given in Definition~\ref{cluster_defn} for all $i,j\in\{1,2,\ldots,m\}$. The set of targets $R_i'$ and $R_j'$ each contain a state that belongs to cluster $j$ and $i$, respectively. All partitions other than $\mathcal{R}$ can be represented by assigning more pairs of states to the wrong clusters. We denote $\pi_{s_0, \mathcal{V}}^* : S \times 2^{\mathcal{V}} \to A$ to be the optimal policy that solves Problem~\ref{probl_ct} with target set $\mathcal{V}\subseteq~S$ and initial state $s_0\in S$. From~\eqref{in_cluster_bounds_ht}~and~\eqref{out_cluster_bounds_ht}, the bounds on cover times are
\begin{equation}
\begin{split}
C_{s_0}^{R_j, \pi^*_{s_0, R_j}} \leq w_2 + (n-1)w_{c} \ \textnormal{for all} \ j\in\{1,2,\ldots,m\},
    \label{cover_bounds_cluster}
    \end{split}
\end{equation}
\begin{equation}
\begin{split}
C_{s_0}^{R_i', \pi^*_{s_0, R_i'}} \geq w_{1} + (n-2)w_c' + w_{l} > w_{1} + w_{l}, \\
C_{s_0}^{R_j', \pi^*_{s_0, R_j'}} \geq w_{1} + (n-2)w_c' + w_{l} > w_{1} + w_{l}.
    \label{cover_bounds_cluster_out}
    \end{split}
\end{equation}
Since hitting time and cover time are random variables, the above bounds also hold in expectation. For $\mathcal{R}$ to be a partition with expected cover time less than that of $\mathcal{R}'$, the following should hold. From~\eqref{cover_bounds_cluster}~and~\eqref{cover_bounds_cluster_out}, if
\begin{equation}
\begin{split}
w_1 + w_{l} \geq w_2 + (n-1)w_{c}, \\
\Rightarrow w_{l} \geq (n-1)w_{c} + (w_2 - w_1),
    \label{condition}
    \end{split}
\end{equation}
then the expected cover time of partition $\mathcal{R}$ is less than that of $\mathcal{R}'$. There may exist two pairs of states in two wrong clusters such that the partition is ${\mathcal{R}''=\{R_1, R_2, \ldots, R_i'', R_j'', \ldots, R_m\}}$ where ${R_i''=\{s_{i}^1, s_{i}^2, \ldots, s_{j}^k, s_{j}^l, \ldots, s_{i}^n\}}$ and ${R_j''=\{s_{j}^1, s_{j}^2, \ldots, s_{i}^k, s_i^l, \ldots, s_{j}^n\}}$. Then,
\begin{equation}
\begin{split}
C_{s_0}^{R_i'', \pi^*_{s_0, R_i''}} \geq w_{1} + (n-3)w_c' + 2w_{l} >  w_{1} + w_{l}, \\ C_{s_0}^{R_j'', \pi^*_{s_0, R_j''}} \geq w_{1} + (n-3)w_c' + 2w_{l} > w_{1} + w_{l}.
    \label{cover_bounds_cluster_out_2}
    \end{split}
\end{equation}
Similarly, from~\eqref{cover_bounds_cluster_out_2}, the condition in~\eqref{condition} should hold for the expected cover time of $\mathcal{R}$ to be less than the expected cover time of any partition with atleast a pair of states in the wrong cluster. Since all partitions other than $\mathcal{R}$ can be represented by assigning more pairs of states to the wrong clusters, $\mathcal{R}$ is the optimal partition for Problem~\ref{probl_multi_ct} if~\eqref{condition} is satisfied. 
\end{proof}

Having proved that the clustered partition $\mathcal{R}$ is the optimal partition, we prove that Algorithm~\ref{alg_part} generates $\mathcal{R}$ under certain conditions.

\begin{thm}
Under Assumption~\ref{ass1}, given an initial partition $\mathcal{P}$ of the target states $\mathcal{V}$ where $|P_i|=n$ for all $i,j~\in~\{1,2,\ldots,m\}$, if there exists a clustered partition $\mathcal{R}$ as in Definition~\ref{cluster_defn} and 
\begin{equation}
w_l > 3nw_c + w_c' + \frac{w_2 - w_1}{2},
    \label{condition_R}
\end{equation}
then Algorithm~\ref{alg_part} will generate the partition $\mathcal{R}$ from an initial partition $\mathcal{P}$.
\label{cluster_alg} 
\end{thm}
\begin{proof}
We are operating on the complete graph $\mathcal{G} = (S, E)$ computed in line~2 of Algorithm~\ref{alg_part}. Let the set of targets associated with agent $i$ be 
\begin{equation}
\begin{split}
P_i = \cup_{j=1}^m Q_i^j \  \textnormal{where} \ Q_i^j \subseteq R_j, \ \cup_{i=1}^m Q_i^j = R_j, \ \sum_{j=1}^m |Q_i^j| = n, \\ Q_i^j \cap Q_k^l = \emptyset, \ \textnormal{for all} \ i, j, k, l~\in~\{1,2,\ldots,m\},i \neq j \neq k \neq l.
\label{init_part_cond}
\end{split}
\end{equation}
The set $Q_i^j$ is the set of target states which belong to cluster $j$ and are associated with agent $i$. Consider a swap operation between two arbitrary sets $P_i, P_k \in \mathcal{P}$ and two target states $s_{i}^j \in Q_i^j, s_k^l \in Q_k^l$. Let the sets after the swap operation be $P_i', P_k'$ and the change in the sum of weights be given by $\Delta W(\mathcal{G}_i)=W(\mathcal{G}_i')~-~W(\mathcal{G}_i)$. Therefore, from~\eqref{swap_G_i},
\begin{equation}
\begin{split}
\Delta W(\mathcal{G}_i) = -\Delta W(\mathcal{G}_i, s_i^j) + \Delta W(\mathcal{G}_i, s_k^l) - \\ w(s_i^j, s_k^l) - w(s_k^l, s_i^j), \\
\Delta W(\mathcal{G}_k) = -\Delta W(\mathcal{G}_k, s_k^l) + \Delta W(\mathcal{G}_k, s_i^j) - \\ w(s_i^j, s_k^l) - w(s_k^l, s_i^j).
\label{change_swap}
\end{split}
\end{equation}
The contribution of a state $s_i^j$ to a sub-graph $\mathcal{G}_i$ can be written from~\eqref{sum_edge_contr} as
\begin{equation}
\begin{split}
\Delta W(\mathcal{G}_i, s_i^j) = \sum_{s_1\in Q_i^j}w(s_1, s_i^j) + \sum_{s_1\in Q_i^j}w(s_i^j, s_1) + \\
\sum_{s_1 \in P_i\setminus~Q_i^j}w(s_1, s_i^j) + \sum_{s_1 \in P_i\setminus~Q_i^j}w(s_i^j, s_1) + w(s_0, s_i^j).
\end{split}
\label{sum_edge_contr_split}
\end{equation}
For ease of writing, let us denote ${|Q_i^j|=n_i^j}$ for all ${i,j\in\{1,2,\ldots,m\}}$. Thus, from~\eqref{ht_initial_bound},~\eqref{sum_edge_contr_split} and~\eqref{out_cluster_bounds_alg},
\begin{equation}
\Delta W(\mathcal{G}_i, s_i^j) \geq 2\left(w_{c}'(n_i^j-1) + w_{l}(n - n_{i}^j)\right) + w_1.
\label{change_bound_1}
\end{equation}
Similarly, from~\eqref{ht_initial_bound},~\eqref{cluster_bound_alg} and~\eqref{sum_edge_contr_split},
\begin{equation}
\Delta W(\mathcal{G}_i, s_k^l) \leq 2\left(w_{c}(n_i^l) + w_{l}'(n - n_{i}^l)\right) + w_2.
\label{change_bound_2}
\end{equation}
Thus, from~\eqref{ht_initial_bound},~\eqref{cluster_bound_alg},~\eqref{out_cluster_bounds_alg} and~\eqref{sum_edge_contr_split}, we can write the bounds for sub-graph $\mathcal{G}_k$ as
\begin{equation}
\begin{split}
\Delta W(\mathcal{G}_k, s_k^l) \geq 2\left(w_{c}'(n_k^l-1) + w_{l}(n - n_{k}^l)\right) + w_1, \\
\Delta W(\mathcal{G}_k, s_i^j) \leq 2\left(w_c(n_k^j) + w_l'(n - n_{k}^j)\right) + w_2.
    \label{change_bound}
    \end{split}
\end{equation}
From~\eqref{change_swap},~\eqref{change_bound_1},~\eqref{change_bound_2} and~\eqref{change_bound},
\begin{equation}
\begin{split}
\frac{\Delta W(\mathcal{G}_i)}{2} \leq (w_l' - w_l)n  - n_i^l (w_l' - w_c) + n_{i}^j(w_l - w_{c}')  + \\ w_c' - w_l +  \frac{w_2 - w_1}{2}, \\
\frac{\Delta W(\mathcal{G}_k)}{2} \leq (w_l' - w_l)n  - n_k^j (w_l' - w_c) + n_{k}^l(w_l - w_{c}') + \\ w_c' - w_l +  \frac{w_2 - w_1}{2}.
    \label{change_swap_bound}
\end{split}
\end{equation}

Using Fig.~\ref{fig:theorem} as an illustration, if we fix $w_l$ to be the smallest weight between states of different clusters, and $w_c$ to be the largest weight between states of same clusters, then the bound for $w_l'$ is
\begin{equation}
w_l' \leq w_l + 2w_c.
    \label{wl2_bound}
\end{equation}
From~\eqref{change_swap_bound} and~\eqref{wl2_bound},
\begin{multline}
    \frac{\Delta W(\mathcal{G}_i)}{2} \leq 
2nw_c - n_i^l (w_l' - w_c) + n_{i}^j(w_l - w_{c}')  + \\ w_c' - w_l +  \frac{w_2 - w_1}{2}.
    \label{bound_result_int_i}
\end{multline}
Since $-w_l'\leq-w_l$, and $w_l - w_c'\leq~w_l$,~\eqref{bound_result_int_i} can be written as
\begin{equation}
    \frac{\Delta W(\mathcal{G}_i)}{2} \leq 
2nw_c - n_i^l(w_l - w_c) + n_i^jw_l + w_c' - w_l +  \frac{w_2 - w_1}{2}.
    \label{bound_result_int1_i}
\end{equation}
Adding and subtracting $n_i^jw_c$ to~\eqref{bound_result_int1_i},
\begin{multline}
    \frac{\Delta W(\mathcal{G}_i)}{2}\leq 
2nw_c - n_i^l(w_l - w_c) + n_i^j(w_l-w_c) + n_i^jw_c + \\ w_c' - w_l +  \frac{w_2 - w_1}{2}.
    \label{bound_result_int2_i}
\end{multline}
Since $n_i^j\leq~n$,~\eqref{bound_result_int2_i} can be written as
\begin{equation}
\begin{split}
    \frac{\Delta W(\mathcal{G}_i)}{2} \leq 
3nw_c - w_l + (n_{i}^j - n_{i}^l)(w_l - w_c) + \\ w_c' +  \frac{w_2 - w_1}{2}.
    \label{bound_result_i}
\end{split}
\end{equation}
Similarly, using the same argument as~\eqref{bound_result_int_i},~\eqref{bound_result_int1_i}, and~\eqref{bound_result_int2_i},
\begin{equation}
\begin{split}
\frac{\Delta W(\mathcal{G}_k)}{2} \leq
3nw_c - w_l + (n_{k}^l - n_{k}^j)(w_l - w_c) + \\ w_c' +  \frac{w_2 - w_1}{2}.
    \label{bound_result_k}
\end{split}
\end{equation}
From~\eqref{bound_result_i} and~\eqref{bound_result_k},
\begin{equation}
\begin{split}
n_{i}^j \leq n_{i}^l \ \textnormal{and} \ n_{k}^l \leq n_{k}^j \ \textnormal{and} \ w_l > 3nw_c + w_c' +  \frac{w_2 - w_1}{2} \Rightarrow \\ \Delta W(\mathcal{G}_i) < 0 \ \textnormal{and} \ \Delta W(\mathcal{G}_k) < 0.
\end{split}
    \label{bound_result}
\end{equation}
From~\eqref{init_part_cond}, since ${\sum_{j=1}^m n_i^j = \sum_{i=1}^m n_i^j= n}$ for all ${i, j\in\{1,2,\ldots,m\}}$, there always exists some $(i,j,k, l)$ such that $n_{i}^j \leq n_{i}^l \ \textnormal{and} \ n_{k}^l \leq n_{k}^j$. If $w_l > 3nw_c + w_c' +  \frac{w_2 - w_1}{2}$, then~\eqref{bound_result} is satisfied, and from Algorithm~\ref{alg_part} and~\eqref{swap_G_i}, the states $s_i^j$ and $s_k^l$ are swapped between $\mathcal{G}_i$ and $\mathcal{G}_k$. After the swap, $n_i^j$ and $n_k^l$ are decremented by one, whereas  $n_i^l$ and $n_k^j$ are incremented by one. Therefore, from~\eqref{bound_result_i} and~\eqref{bound_result_k}, $\Delta W(\mathcal{G}_i') < \Delta W(\mathcal{G}_i)$ and $\Delta W(\mathcal{G}_k') < \Delta W(\mathcal{G}_k)$ where $\mathcal{G}_i'$ and $\mathcal{G}_k'$ are the sub-graphs after the swap. Hence, the states are swapped until $n_i^j = n_k^l = 0$ or $n_i^l = n_k^j = n$. If $n_i^l = n$, then all the states of cluster $l$ are associated with agent $i$. If $n_i^j = 0$, then all the states $s_i^j\in Q_i^j$ which belong to cluster $j$ and were initially associated with agent $i$ are now with agent $k$. The states are swapped for all ${i, j, k, l\in\{1,2,\ldots,m\}}$ until the boundary cases of $n_i^j = n_k^l = 0$ or $n_i^l = n_k^j = n$ are reached. Since $\sum_{j=1}^m n_i^j = \sum_{i=1}^m n_i^j= n$ for all $i, j~\in~\{1,2,\ldots,m\}$, all the states of each cluster $j$ are associated with only one agent $i$, thus resulting in the partition $\mathcal{R}$. An agent cannot have two clusters because the initial partition has $n$ targets associated with each agent which is the same as the number of targets in each cluster. Hence, only swap operations happen between agents because a transfer operation increases $W(\mathcal{G}_i)$ for some agent $i$.
\end{proof}

 Therefore, Algorithm~\ref{alg_part} generates optimal partitions ${\mathcal{P}^* = \mathcal{R}}$ for Problem~\ref{probl_multi_ct} if the bounds on the hitting time of the states in the MDP independently satisfy the conditions in Theorem~\ref{cluster_is_optimal} and Theorem~\ref{cluster_alg}. One such MDP was illustrated in Fig.~\ref{fig:theorem}.

We assumed in Theorem~\ref{cluster_alg} that the number of targets assigned to each agent in the initial partition are all equal. If we do not have any prior knowledge about the MDP dynamics and location of targets, one naive way to partition would indeed be to assign equal number of targets to each agent. We expect a partition with equal sized target assignments to be optimal for Problem~\ref{probl_multi_ct} if the optimal hitting times and cover times are independent and uniformly distributed. However, this question formally remains open for future work. 

 We also assume in our work that the clusters defined in Definition~\ref{cluster_defn} have equal number of targets. We believe the ideas presented in our proof of Theorem~\ref{cluster_alg} can be extended to clusters with different number of targets. We could bound the change in sum of weights $\Delta W\left(\mathcal{G}_i\right), \Delta W\left(\mathcal{G}_k\right)$ for every pair of sub-graphs $\left(\mathcal{G}_i, \mathcal{G}_k\right)$ in terms of the number of targets $n_j$ in cluster $j$ where $i,j,k\in\{1,2,\ldots,m\}$. We could then derive conditions for a swap operation between the sub-graphs similar to~\eqref{bound_result}. We reserve this more detailed analysis for future work.     
 
Though Theorem~\ref{cluster_is_optimal} and Theorem~\ref{cluster_alg} apply only to MDPs with clustered targets, we show in the subsequent section with numerical experiments that Algorithm~\ref{alg_part} generates optimal partitions $\mathcal{P}^*$ which solves Problem~\ref{probl_multi_ct} for some MDPs that not necessarily have clustered targets. We also validate the algorithms for single agent described in Section~\ref{sec:exact_sol_sec} and Section~\ref{sec:sub_opt_single} using numerical experiments.

\section{Numerical Results}
\label{sec:numerical}

In this section, we validate our algorithms on different environments for both single-agent and multi-agent scenarios. Since previous work was on heuristics for single agent planning on graphs~\cite{heur_TSP}, we compare the performance and results of our single agent planning algorithm on random graphs. Our work is primarily motivated by stochastic dynamics and hence we present the numerical results of our algorithms on random MDPs. We also validate our algorithm on more realistic gridworld environment motivated by ocean dynamics.

\subsection{Random graphs}

In this section, we implement our heuristic procedure on graphs for single agent path planning and partitioning of targets to multiple agents.

\subsubsection{Single agent}
\label{res_single_graph}

In this section, we present our results for the single agent case. The paths produced by Algorithm~\ref{alg_heur_ct} to cover all states of some deterministic MDPs are given in Fig.~\ref{fig:opt_graphs_alg2}. We compare the performance of Algorithm~\ref{alg_heur_ct} and the nearest neighbor heuristic~\cite{heur_TSP}. We use policy iteration to compute the optimal cover time, but we could use any algorithm that solves the SSP problem to optimality on the product MDP~\cite{SSP_opt}. We compute the average value of the cover time by performing 1000 runs of the nearest neighbor heuristic. However, we do not need to run policy iteration and Algorithm~\ref{alg_heur_ct} 1000 times to get the cover time; instead, just running it once produces the cover time of the deterministic policy. We use $\varepsilon=10^{-20}$ and $\gamma=0.01$ for Algorithm~\ref{alg_heur_ct}. We choose such an $\varepsilon$ close to zero so that the computation of value function terminates after some reasonable number of iterations, while at the same time the computed approximate value function is close to the actual value function. If $\varepsilon=0$, the iterative procedure might not stop in reasonable time because of floating point errors, although the value function has converged. We choose $\gamma = 0.01$ in Algorithm~\ref{alg_heur_ct} so that the algorithm incorporates some look-ahead on how good the states are in future time. If $\gamma=0$, there is no look-ahead as described in Section~\ref{sec:sub_opt_single}. Hence, Algorithm~\ref{alg_heur_ct} would perform exactly as the nearest neighbor heuristic. As we increase $\gamma$, the rate of convergence to the optimal value function at each time step decreases considerably without much gain in the overall optimality of the cover time for the MDPs presented in this section.

\begin{figure}[!t]
     \centering
     \begin{subfigure}[h]{0.3\linewidth}
         \centering
         \includegraphics[width=0.53\linewidth]{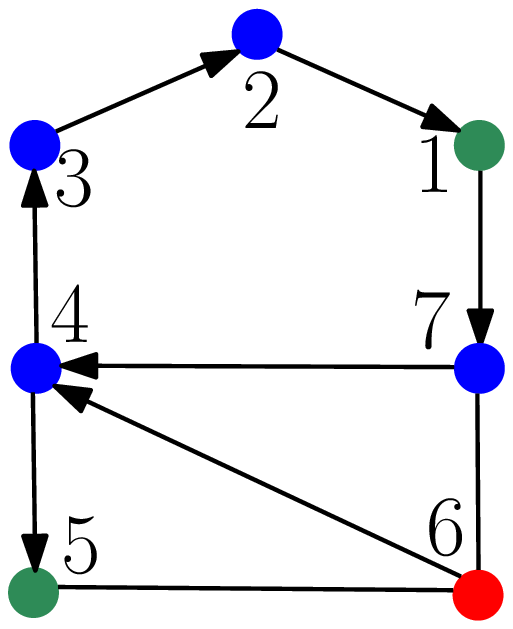}
         \caption{$6\small{\rightarrow}4\small{\rightarrow}3\small{\rightarrow}2\small{\rightarrow}1\small{\rightarrow}7\small{\rightarrow}4 \newline \small{\rightarrow}5$ is suboptimal.}
         \label{fig:rand_graph_1}
     \end{subfigure}
     \hspace{1cm}
     \begin{subfigure}[h]{0.3\linewidth}
         \centering
         \includegraphics[width=0.57\linewidth]{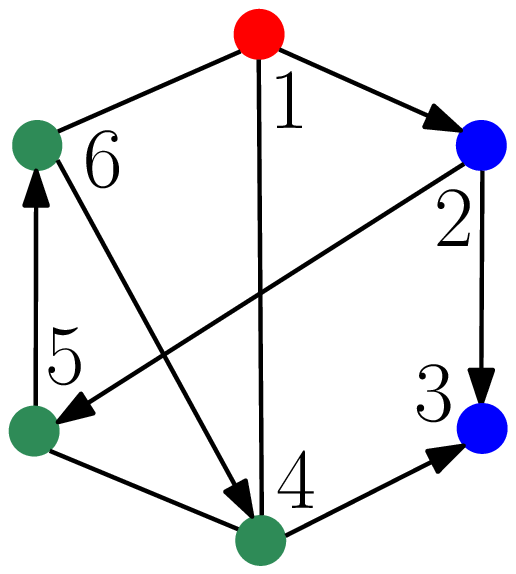}
         \caption{$1\small{\rightarrow}2\small{\rightarrow}5\small{\rightarrow}6\small{\rightarrow}4\small{\rightarrow}3$ is optimal.}
         \label{fig:VI_heur_opt_2}
     \end{subfigure}
     \\
      \begin{subfigure}[h]{\linewidth}
         \centering
         \includegraphics[width=0.6\linewidth]{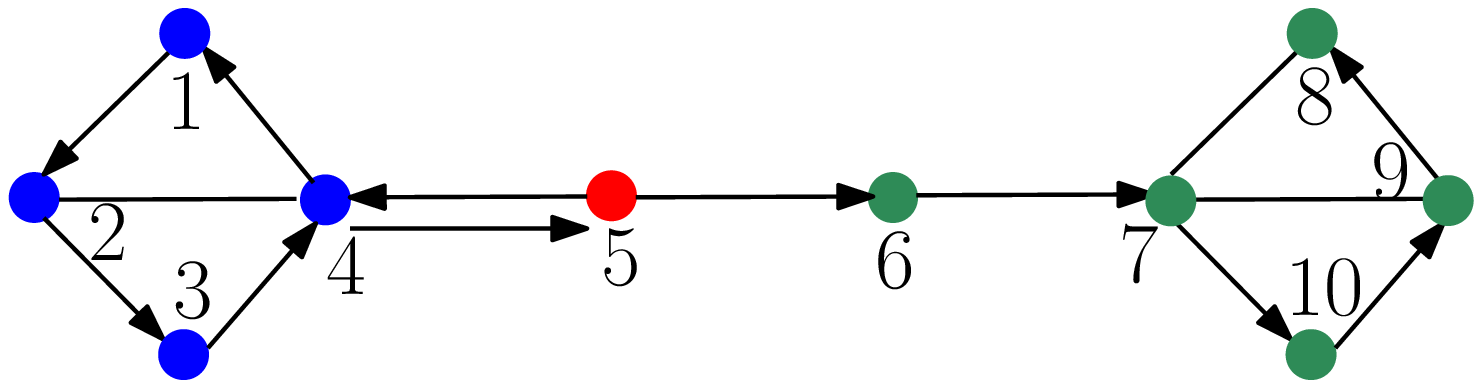}
         \caption{$5\small{\rightarrow}4\small{\rightarrow}1\small{\rightarrow}2\small{\rightarrow}3\small{\rightarrow}4\small{\rightarrow}5\small{\rightarrow}6\small{\rightarrow}7\small{\rightarrow}10\small{\rightarrow}9\small{\rightarrow}8$ is optimal.}
         \label{fig:VI_heur_opt}
     \end{subfigure}
     \caption{Paths produced by Algorithm~\ref{alg_heur_ct} to cover all states of some graphs and illustration of Algorithm~\ref{alg_part} with $m=2$ agents. The initial state is in red, and the states assigned to each agent are shown in green and blue, respectively.}
        \label{fig:opt_graphs_alg2}
\end{figure}


We choose the graphs in Fig.~\ref{fig:opt_graphs_alg2} to be significantly different than those in Fig.~\ref{fig:opt_graphs} to show that Algorithm~\ref{alg_heur_ct} is optimal for a larger class of graphs like those in Figs.~\ref{fig:VI_heur_opt_2} and~\ref{fig:VI_heur_opt}. The runtime of Algorithm~1 for the graphs in Fig.~\ref{fig:opt_graphs_alg2} is 100-1000 times faster than the optimal algorithm, with just a $5\%$ suboptimality on average. Though the nearest neighbour heuristic is 10~times faster than our heuristic, the average suboptimality is approximately $72\%$.


To illustrate our heuristic on more general graphs, we simulate 1000 runs of the nearest neighbor heuristic and Algorithm~\ref{alg_heur_ct}, with $\varepsilon=10^{-20}, \gamma = 0.01$, for random connected graphs. We present the cover time values and the average run-time for 1000 system runs of Algorithm~\ref{alg_heur_ct} and nearest neighbor heuristic in Table~\ref{graph_results}. The policy iteration procedure consumes an immense amount of runtime in MATLAB to compute the optimal cover time for larger than 11 target states, which we denote as `timeout' in Table~\ref{graph_results}. However, Algorithm~\ref{alg_heur_ct} and the nearest neighbor heuristic run much quicker than the optimal algorithm. The nearest neighbor heuristic naturally has better runtime than Algorithm~\ref{alg_heur_ct}, however, it yields significantly worse cover times. On the other hand, the cover times of Algorithm~\ref{alg_heur_ct} are always optimal or almost-optimal, whereas the average runtime is still on the order of 100-1000 times faster than the optimal algorithm. The average variance of the cover time for nearest neighbour heuristic is $54.84$ for the graphs in Table~\ref{graph_results}. We performed all the numerical experiments using MATLAB R2020a on a computer with an Intel Core i7 2.6 GHz processor and 16GB RAM.

\bgroup
\def\arraystretch{1}
\centering
\begin{table}[!t]
 \captionsetup{justification=centering}
\caption{Cover time comparison on random graphs for a single agent.}
\centering
\begin{tabular}{||p{0.4cm}|p{0.25cm}|p{0.75cm}|p{0.9cm}|p{0.5cm}|p{0.75cm}|p{0.75cm}|p{0.75cm}||}
\hline
\multicolumn{2}{||c|}{Graph} & \multicolumn{2}{|p{2cm}|}{Optimal algorithm (policy iteration)} & \multicolumn{2}{|p{1.5cm}|}{Algorithm~\ref{alg_heur_ct}} & \multicolumn{2}{|p{2cm}||}{Nearest neighbor heuristic}\\
\hline
$|S|$ & $|\mathcal{V}|$ & Optimal cover time & \textit{Runtime (sec)} & Cover time & \textit{Runtime (sec)} & Average cover time & \textit{Average runtime (sec)}\\
\hline
\hline
91	& 10 & 10 &	547.948 & 11 &	3.389 &	12.878  & 0.005	\\
\hline
50 & 10 & 10 & 192.798	& 10 & 0.243 & 12.489 &	0.003 \\
\hline
62 & 8 & 9 & 27.946	& 9	& 0.946	& 13.403 & 0.003 \\
\hline
97 & 10 & 11 & 726.563	& 12 & 0.616 & 15.977 & 0.004\\
\hline
200 & 10 & 11 & 4112.927 & 11 & 6.515 & 13.835 & 0.014\\
\hline
150 & 8 & 11 & 263.947 & 11	& 2.537 & 13.389  & 0.011\\
\hline
130 & 9 & 9	& 357.842 & 10 & 1.771 & 11.196  & 0.007\\
\hline
170 & 10 & 10 & 2897.701 & 10 & 3.709 & 13.483 & 0.013\\
\hline
250 & 9 & 10 & 2086.298	& 11 & 11.993 & 12.797  & 0.019\\
\hline
200 & 11 & 11 & 41473.81 & 11 & 6.94 & 13.484 & 0.013\\
\hline
1000 & 9 & 12 & 9167.786 & 13 & 15.894 & 26.087  & 0.09\\
\hline
180 & 12 & N/A & timeout & 12 & 7.427 & 15.988  & 0.017\\
\hline
200 & 15 & N/A & timeout & 16 & 9.271 & 21.232  & 0.02 \\
\hline
500 & 50 & N/A & timeout & 55 & 35.153 & 101.084  & 0.051\\
\hline
1000 & 100 & N/A & timeout & 113 & 78.472 & 212.96  & 0.101\\
\hline
500 & 80 & N/A & timeout & 89 & 70.285 & 185.17  & 0.06\\
\hline
\end{tabular}
\label{graph_results}
\end{table}
\egroup

\subsubsection{Multiple agents}
\label{rand_graph_sec_part}

In this section, we present our results of the partitioning algorithm for the multi-agent case. We use the greedy vertex $m-$center algorithm~\cite{greedy_k} to generate the initial partition $\mathcal{P}$ in line~3 of Algorithm~\ref{alg_part} for $m$ agents.

In Fig.~\ref{fig:opt_graphs_alg2}, we present the partitions generated by Algorithm~\ref{alg_part} for the same graphs discussed in Section~\ref{res_single_graph} with $m=2$ agents. In Fig.~\ref{fig:part_graphs}, we illustrate Algorithm~\ref{alg_part} for random graphs with $m=3$ agents and $\mathcal{V}=S$. In Table~\ref{part_graph_results}, we compare the cover time $\max_{i} C^{P_i, \pi_i^*}_{s_0}$ between the optimal partition $\mathcal{P}^*$ for Problem~\ref{probl_multi_ct} and the partition $\mathcal{P}$ generated by Algorithm~\ref{alg_part} for the graph scenarios in Fig.~\ref{fig:opt_graphs_alg2} and Fig.~\ref{fig:part_graphs}. We use a naive method to obtain the optimal partition $\mathcal{P}^*$ by brute-force search of the optimal cover time of all possible subsets of $\mathcal{V}$. The optimal cover time can be computed using policy iteration or any other algorithm that solves the SSP problem to optimality on a product MDP~\cite{SSP_opt}. From Table~\ref{part_graph_results}, the partition generated by Algorithm~\ref{alg_part} is optimal for graphs in Figs.~\ref{fig:VI_heur_opt_2},~\ref{fig:VI_heur_opt},~\ref{fig:part_graph_2} and~\ref{fig:part_graph_4}. The average runtime of brute-force search is $61.292$~seconds, whereas Algorithm~\ref{alg_part} consumes $0.01$~seconds on average for the graphs in Table~\ref{part_graph_results}. 


\begin{figure}[!b]
      \centering
     \begin{subfigure}[h]{0.3\linewidth}
        \centering
        \includegraphics[width=\linewidth]{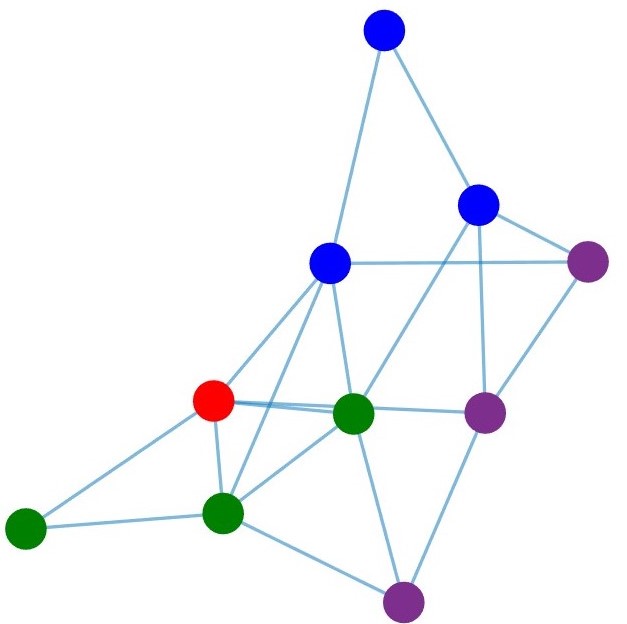}
        \caption{Random graph~1.}
        \label{fig:part_graph_1}
     \end{subfigure}     
     \hspace{1cm}
      \centering
     \begin{subfigure}[h]{0.3\linewidth}
        \centering
        \includegraphics[width=\linewidth]{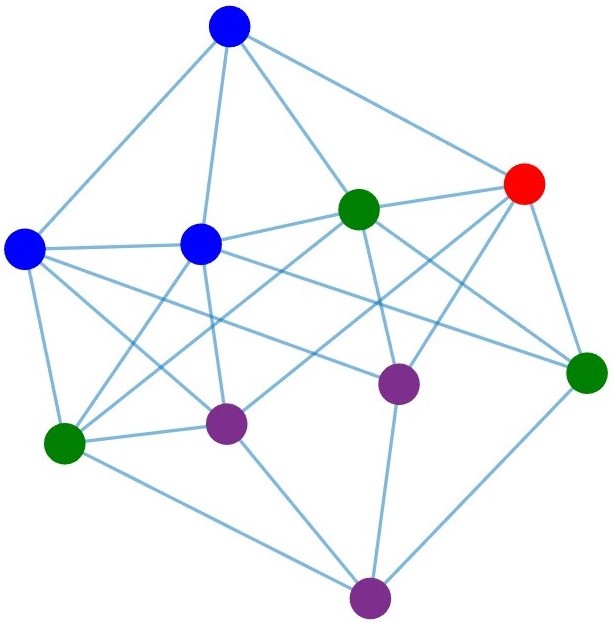}
        \caption{Random graph~2.}
        \label{fig:part_graph_2}
     \end{subfigure}
     \\
       \centering
     \begin{subfigure}[h]{0.3\linewidth}
        \centering
        \includegraphics[width=\linewidth]{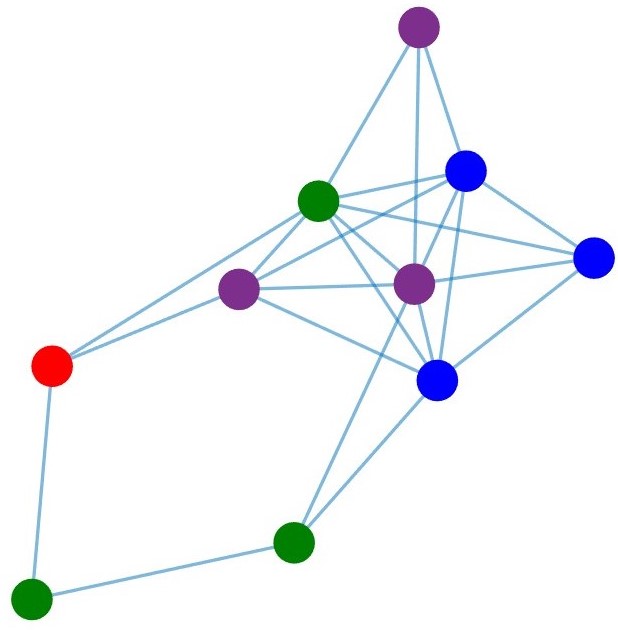}
        \caption{Random graph~3.}
        \label{fig:part_graph_3}
     \end{subfigure}
     \hspace{1cm}
      \centering
     \begin{subfigure}[h]{0.3\linewidth}
        \centering
        \includegraphics[width=\linewidth]{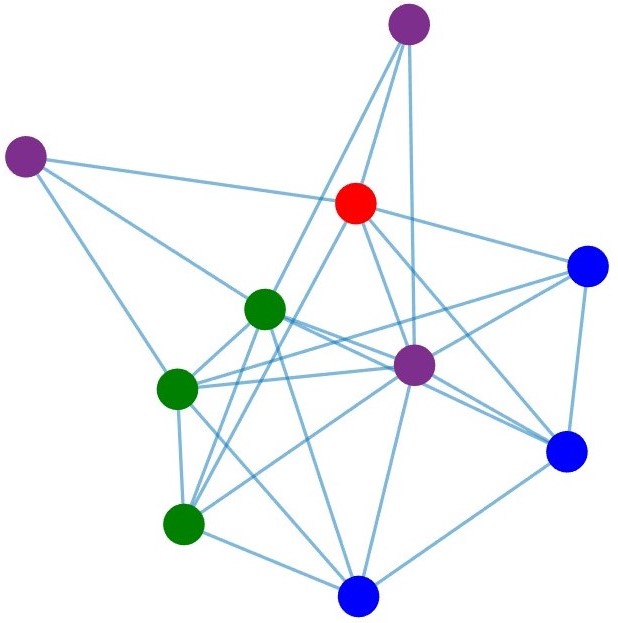}
        \caption{Random graph~4.}
        \label{fig:part_graph_4}
     \end{subfigure}
        \caption{Illustration of Algorithm~\ref{alg_part} on random graph scenarios with $|\mathcal{V}| = 10$ for $m=3$ agents. The initial state is in red, and the states assigned to each agent are shown in green, blue and violet, respectively.}
        \label{fig:part_graphs}
\end{figure}

\subsection{Random MDPs}

In this section, we implement our algorithms on MDPs for single agent path planning and target assignment to multiple agents.

\subsubsection{Single  agent}

In this section, we validate Algorithm~\ref{alg_heur_ct} for the single-agent case with $\varepsilon=10^{-20}$ and $\gamma=0.01$, on general MDPs. We test the algorithms on MDPs where the number of states and targets are selected at random, and for each state $s\in~S$ and action $a\in A$, the transition probabilities $\mathcal{T}(s,a,s')$ form a discrete uniform distribution with all $s'\in~S$ as the support. We select the initial state and the set to be covered at random. We present the optimal expected cover time computed using policy iteration, and the average cover time and average runtime for 1000 system runs of Algorithm~\ref{alg_heur_ct} in Table~\ref{MDP_results}. The average variance of the cover time obtained by Algorithm~\ref{alg_heur_ct} for the MDPs in Table~\ref{MDP_results} is $10.48$. The average suboptimality of Algorithm~\ref{alg_heur_ct} for the MDPs in Table~\ref{MDP_results} is 19.658\%. This suboptimality is offset by a considerable reduction in runtime. The average runtime of the MDPs in Table~\ref{MDP_results} for the optimal algorithm using policy iteration is 13404.225 sec, and Algorithm~\ref{alg_heur_ct} is 22.262 sec.

\bgroup
\def\arraystretch{1}
\begin{table}[!t]
\centering
\captionsetup{justification=centering}
\caption{Cover time comparison for multiple agents for the graphs in Fig.~\ref{fig:opt_graphs_alg2} and Fig.~\ref{fig:part_graphs}.}
\begin{tabular}{||c||p{2cm}|p{0.8cm}|p{2cm}|p{0.8cm}||} 
 \hline
Graph & \multicolumn{2}{|c|}{Optimal partition $\mathcal{P}^*$} & \multicolumn{2}{|c||}{Partition $\mathcal{P}$ by Algorithm~\ref{alg_part}} \\ [0.5ex]
\cline{2-5}
& Cover time $\max_{i} C^{P_i,\pi_i^*}_{s_0}$  & \textit{Runtime (sec)} & Cover time $\max_{i} C^{P_i,\pi_i^*}_{s_0}$ & \textit{Runtime (sec)} \\ 
 \hline\hline
   Fig.~\ref{fig:rand_graph_1} & 3 &  31.21 & 4 & 0.003\\
 \hline
   Fig.~\ref{fig:VI_heur_opt_2} & 3 & 32.13 & 3 & 0.003 \\
 \hline
   Fig.~\ref{fig:VI_heur_opt} & 5 & 40.213 & 5 & 0.005\\
 \hline
 Fig.~\ref{fig:part_graph_1} & 3 & 81.213 & 4 & 0.019 \\
 \hline
 Fig.~\ref{fig:part_graph_2} & 3 & 80.15 & 3 & 0.013\\
 \hline
Fig.~\ref{fig:part_graph_3} & 3 & 79.91 & 4& 0.009 \\
 \hline
 Fig.~\ref{fig:part_graph_4} & 3 & 84.217 & 3 & 0.021\\
 \hline
\end{tabular}
\label{part_graph_results}
\end{table}
\egroup

\bgroup
\def\arraystretch{1}
\centering
\begin{table}[!b]
\captionsetup{justification=centering}
\caption{Cover time comparison on random MDPs for a single agent.}
\centering
\begin{tabular}{||p{0.4cm}|p{0.2cm}|p{0.9cm}|p{1.1cm}|p{1cm}|p{1cm} | p{1cm}||}
\hline
\multicolumn{2}{||c|}{MDP} & \multicolumn{2}{|p{2.3cm}|}{Optimal algorithm (policy iteration)} & \multicolumn{3}{|c||}{Algorithm~\ref{alg_heur_ct}}\\
\hline
$|S|$  & $|\mathcal{V}|$ & Expected cover time & \textit{Runtime (sec)} & Average cover time & Variance of cover time & \textit{Average runtime (sec)} \\
\hline
\hline
50 & 10 & 36.732 & 1573.755	& 38.98	& 5.949 & 0.141	\\
\hline
70  & 8 & 45.394	& 395.643 & 49.253 & 7.801 & 0.19\\
\hline
80  & 10 & 37.603 & 4166.8411 & 42.91 & 5.642 & 0.193\\
\hline
100  & 9 & 51.517 & 1850.285	& 60.878 & 11.133 & 0.484\\
\hline
100  & 10 & 52.7734	& 6400.20	& 57.847	& 8.397 & 0.514\\
\hline
200  & 10 & 68.355 & 59338.536	& 86.148	& 14.679 & 5.266\\
\hline
200  & 9 & 56.348 & 3862.91 & 74.056 & 6.173 & 2.661\\
\hline
150  & 8 & 59.714 & 741.04 & 72.89 & 10.234 & 0.17\\
\hline
170  & 10 & 73.287 & 43615.027	& 93.675	& 18.311 & 1.035\\
\hline
120  & 10 & 57.024 & 14654.176 & 73.824 & 12.493 & 0.772\\
\hline
1000  & 9 & 90.067 & 10081.21 & 111.59 & 12.64 & 5.823\\
\hline
500  & 10 & N/A & timeout & 110.659 & 14.756 & 13.96\\
\hline
500  & 50 & N/A & timeout & 859.12 & 19.615 & 47.965\\
\hline
1000  & 12 & N/A & timeout & 643.974 & 18.531 & 31.039\\
\hline
\end{tabular}
\label{MDP_results}
\end{table}
\egroup

\subsubsection{Multiple agents}
\label{num_mdp_multi}

In this section, to illustrate the heuristic partitioning procedure on more general MDPs, we implement Algorithm~\ref{alg_part} on the random MDPs presented in Table~\ref{MDP_results}. In Table~\ref{part_MDP_results}, we present the optimal expected cover time computed using the naive method described in Section~\ref{rand_graph_sec_part}, and the expected cover time of the partitions generated by Algorithm~\ref{alg_part} for random MDPs with $m=3$ agents. We implement Algorithm~\ref{alg_heur_ct} for each agent using the partition generated by Algorithm~\ref{alg_part} as well as the optimal partition. We present the average cover time and average runtime for 1000 system runs of Algorithm~\ref{alg_heur_ct} in Table~\ref{part_MDP_results}. The cover time for one run of Algorithm~\ref{alg_heur_ct} for multiple agents is computed as $\max_i(C_{s_0}^{P_i})$ where $P_i$ are the targets assigned to agent $i$ and $C_{s_0}^{P_i}$ is the cover time of agent~$i$. 

{We present the results only for $3$ agents because of the saturation of cover time with a larger number of agents for the number of targets in Table~\ref{part_MDP_results}. In Fig.~\ref{fig:magent_scale}, we plot the ratio of the average cover time for the multi-agent case to the single-agent case for three MDP scenarios, when the number of agents are increased. The average cover time is computed for 1000 system runs of Algorithm~\ref{alg_heur_ct} on the partition $\mathcal{P}$ generated by Algorithm~\ref{alg_part}. The average runtime of Algorithm~\ref{alg_part} marginally increases with an increase in the number of agents $m$, but it also eventually saturates. We keep the full analysis of the effect of number of agents on the cover time and performance of Algorithm~\ref{alg_part} for future work.

\begin{figure}[!b]
         \centering        \includegraphics[trim={0cm 0cm 0cm 0cm},clip,width=0.8\linewidth]{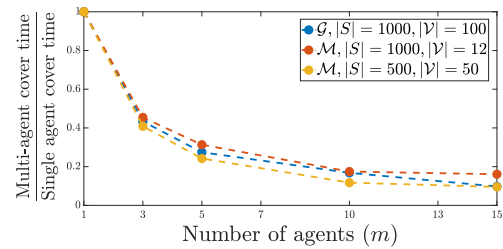}
         \caption{Effect of saturation of average cover time with increasing number of agents when Algorithm~\ref{alg_heur_ct} is implemented on the partition $\mathcal{P}$ by Algorithm~\ref{alg_part}.}
         \label{fig:magent_scale}
\end{figure}

From Table~\ref{part_MDP_results}, the expected cover time of the optimal single-agent algorithm on the partition $\mathcal{P}$ generated by Algorithm~\ref{alg_part} is either optimal or almost-optimal when compared to the optimal partition $\mathcal{P}^*$. The average cover time is similar when Algorithm~\ref{alg_heur_ct} is implemented on the partition generated by Algorithm~\ref{alg_part} and the optimal partition. Hence, the partition generated by Algorithm~\ref{alg_part} is the same or similar to the optimal partition. Since Algorithm~\ref{alg_heur_ct} is implemented independently for each agent in the multi-agent case, there is a slight increase of suboptimality in the cover time of Algorithm~\ref{alg_heur_ct} when compared to the single agent case. Though the average cover time when implementing Algorithm~\ref{alg_heur_ct} is more than one half of the optimal cover time for $m=3$ agents, the runtime of our heuristic partitioning and path planning procedure is $10^5$ times faster than the optimal partitioning and optimal path planning method.

\subsection{Ocean dynamics}

In this section, we implement our heuristic procedures for single agent path planning, and assignment of multiple-targets to multiple agents on stochastic gridworlds modeled by realistic ocean currents~\cite{oceans}.

\subsubsection{Single agent}
\label{num_single_ocean}

In this section, we simulate Algorithm~\ref{alg_heur_ct} with transition dynamics motivated by an autonomous underwater vehicle operating in an oceanic environment~\cite{oceans}, modeled as a $20 \times 20$ grid. The vehicle's transitions are to move one step north, west, south, or east of its current state, except at the edges of the environment. The set to be covered constitutes regions associated with formation and evolution of algal blooms~\cite{oceans}. The uncertainty and high variability of ocean currents in space and time demand the use of stochastic transition dynamics. We use the MDP environment in~\cite{MDP_env} to generate realistic transition probabilities from a Gaussian distribution of ocean models~\cite{oceans}. We illustrate Problem~\ref{probl_ct} on the resulting gridworld using Algorithm~\ref{alg_heur_ct} with ${\varepsilon = 10^{-20}, \gamma = 0.4}$ in Fig.~\ref{fig:grid_sim}.  We increase the value of $\gamma$ compared to previous experiments because there is a considerable decrease in suboptimality for the MDPs considered in this section, even though there is a slight decrease in rate of convergence to the optimal value function at each time step. The cover time for the path presented in Fig.~\ref{fig:grid_sim} is 88. The two ``blips'' in the final stretch of the path before completing the mission are a consequence of stochastic dynamics. The average cover time for 1000 runs of Algorithm~\ref{alg_heur_ct} for the problem in Fig.~\ref{fig:grid_sim} is 84.87 with a variance of $21.267$, and the average runtime is 9.0119 sec. The optimal expected cover time for the scenario in Fig.~\ref{fig:grid_sim} is 71.8534 which is computed using policy iteration with a runtime of 5760 sec.

\begin{figure}[!b]
         \centering        \includegraphics[width=0.5\linewidth]{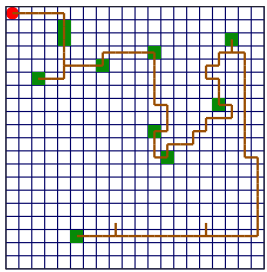}         \caption{A representative path produced by Algorithm~\ref{alg_heur_ct} to cover 10 states on a $20 \times 20$ gridworld with stochastic dynamics. The red dot is the initial state, green squares are the targets, and the path of the agent is denoted by brown lines.}
         \label{fig:grid_sim}
\end{figure}

\subsubsection{Multiple agents}

In this section, we implement Algorithm~\ref{alg_part} to assign targets to agents motivated by a team of autonomous underwater vehicles operating in an ocean~\cite{oceans}, with the same environment and agents used in Section~\ref{num_single_ocean}. We illustrate Problem~\ref{probl_multi_ct} on the resulting gridworlds using Algorithm~\ref{alg_part} in Fig.~\ref{fig:sim_part_all} and Fig.~\ref{fig:sim_part_all_impl}.

In Fig.~\ref{fig:sim_part_all}, we depict the assignment of targets to agents on a $10\times~10$ stochastic gridworld using Algorithm~\ref{alg_part}. The scenario in Fig.~\ref{fig:opt_alg_part_10grid_9T} has clustered target states, and the initial state of the agents is approximately equidistant from all the clusters. Therefore, Algorithm~\ref{alg_part} has naturally generated the clustered partition which is also verified to be the optimal partition. The optimal partition is obtained by doing a brute-force search of all possible optimal expected cover times computed using policy iteration. In Fig.~\ref{fig:opt_alg_part_10grid_10T},
the initial state and the set of target states are selected at random. Algorithm~\ref{alg_part} has again generated the optimal partition for the scenario in Fig.~\ref{fig:opt_alg_part_10grid_10T} where the optimality is verified by brute-force search.

\begin{figure}[!b]
     \centering
      \begin{subfigure}[h]{0.4\linewidth}
        \centering
        \includegraphics[width=\textwidth]{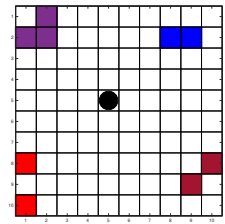}
        \caption{$m=4$ agents with $|\mathcal{V}|=9$ targets.}
        \label{fig:opt_alg_part_10grid_9T}
     \end{subfigure}
      \hfill
       \begin{subfigure}[h]{0.4\linewidth}
        \centering
        \includegraphics[width=\textwidth]{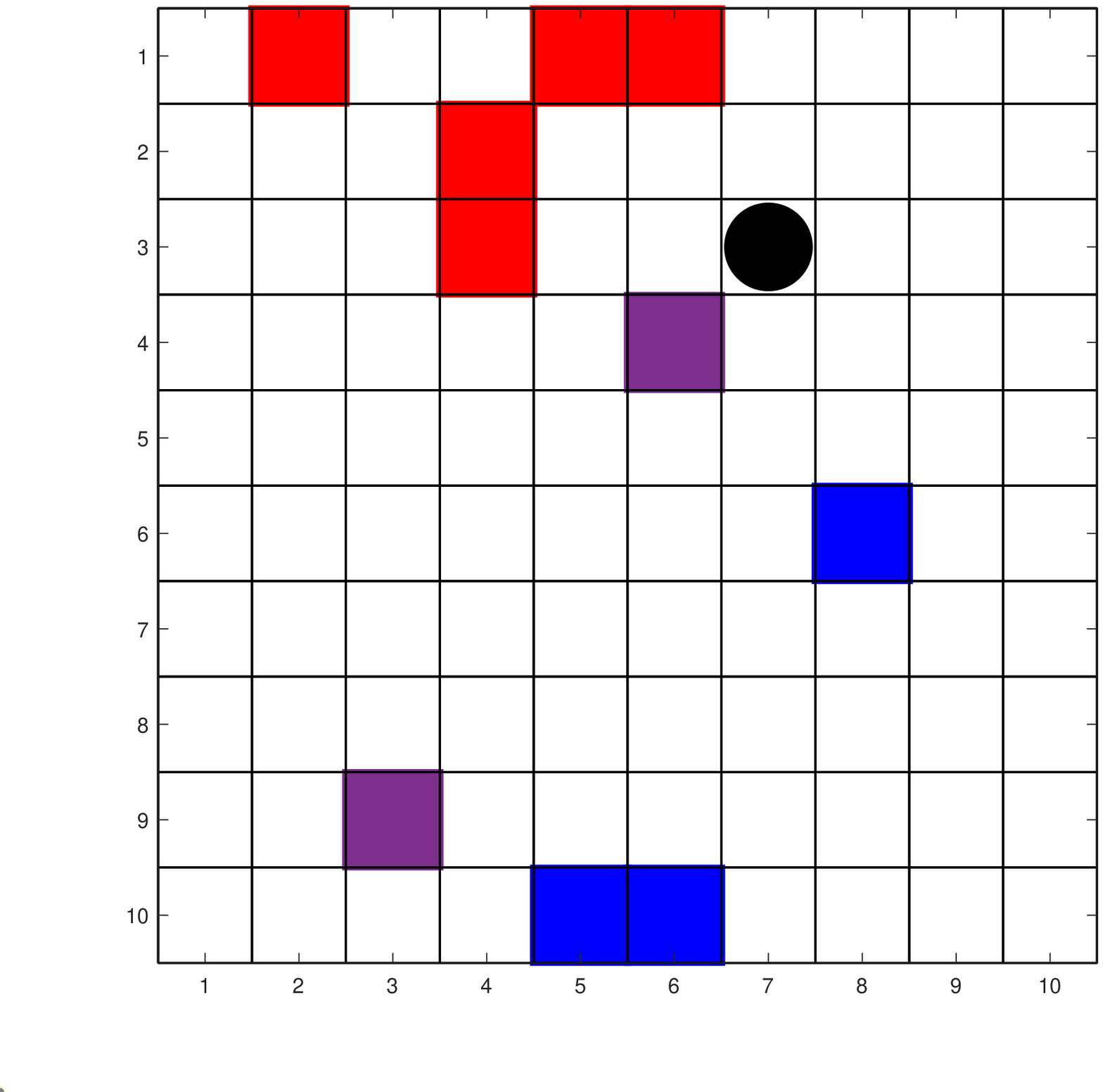}
        \caption{$m=3$ agents with $|\mathcal{V}|=10$ targets.}
        \label{fig:opt_alg_part_10grid_10T}
     \end{subfigure}
     \caption{Illustration of Algorithm~\ref{alg_part} on ${10\times10}$ gridworld environments with stochastic transition dynamics. The black dot is the initial state, and the target assignment for each agent are denoted by blue, violet, red, and brown squares.}
\label{fig:sim_part_all}
\end{figure}

\begin{figure}[!b]
    \begin{subfigure}[h]{0.45\linewidth}
        \centering
        \includegraphics[width=\textwidth]{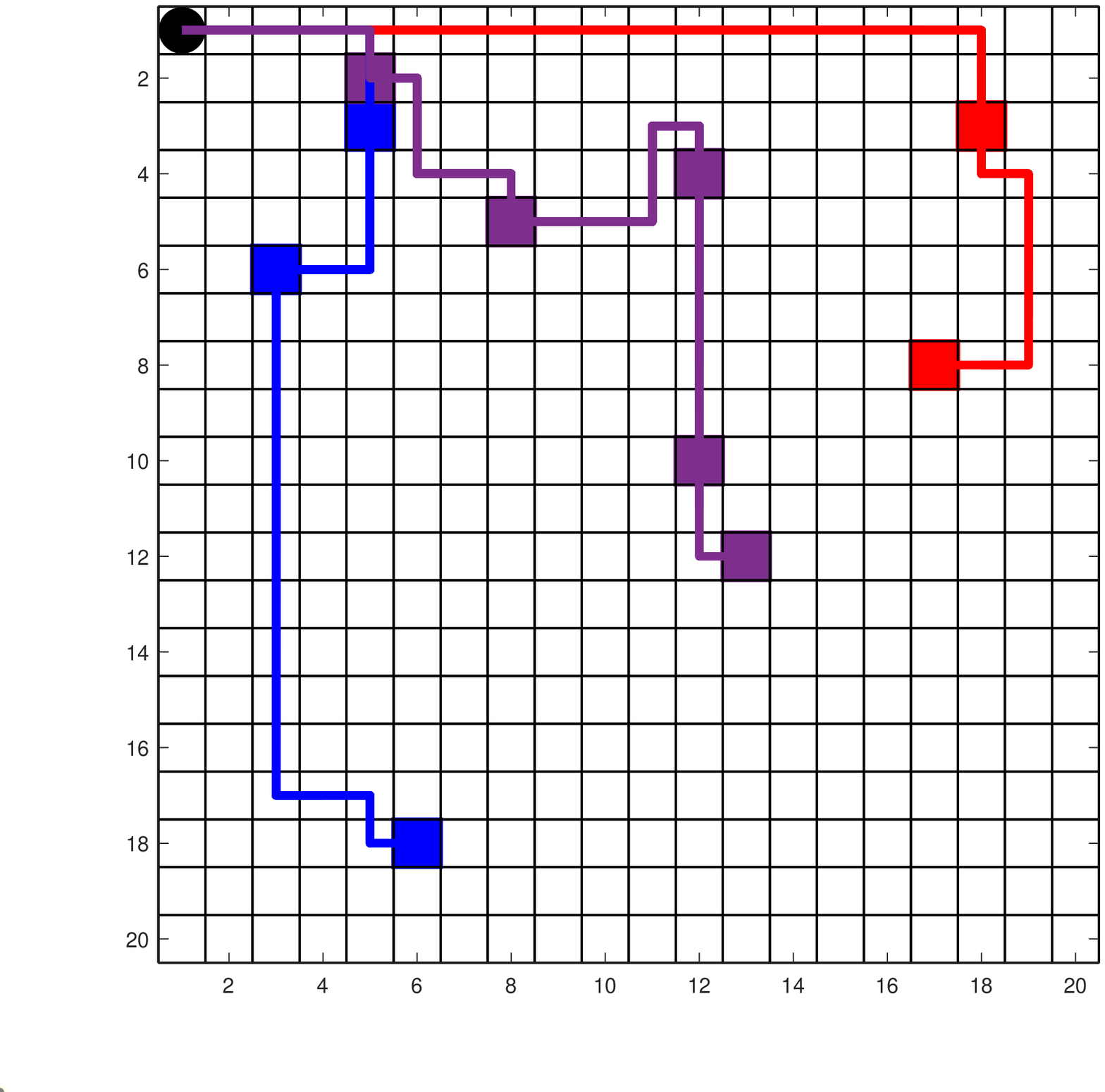}
        \caption{$m=3$ agents with $|\mathcal{V}|=10$ targets.}
        \label{fig:opt_alg_part_pranay_grid}
     \end{subfigure}
    \hfill
     \begin{subfigure}[h]{0.45\linewidth}
        \centering
        \includegraphics[width=\textwidth]{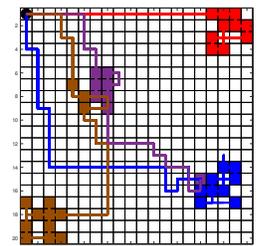}
        \caption{$m=4$ agents with $|\mathcal{V}|=40$ targets.}
        \label{fig:sim_part_heur}
     \end{subfigure}
     \caption{A representative path generated by Algorithm~\ref{alg_heur_ct} after the partitioning by Algorithm~\ref{alg_part} on ${20\times20}$ gridworld environments with stochastic dynamics. The black dot is the initial state, the path and targets for each agent are denoted by blue, violet, red and brown lines and squares, respectively.}
\label{fig:sim_part_all_impl}
\end{figure}

In Fig.~\ref{fig:sim_part_all_impl}, we illustrate the assignment of targets to agents on $20\times~20$ stochastic gridworlds using Algorithm~\ref{alg_part} and also implement Algorithm~\ref{alg_heur_ct} with $\varepsilon = 10^{-20}, \gamma = 0.7$ to synthesize each agent's policy to visit the targets assigned to them. In Fig.~\ref{fig:opt_alg_part_pranay_grid}, we depict the same scenario presented in Fig.~\ref{fig:grid_sim} but with $m=3$ agents. Algorithm~\ref{alg_part} has generated the optimal partition for Fig.~\ref{fig:opt_alg_part_pranay_grid}. The optimal partition can be verified by a brute force method which requires 1586~sec whereas Algorithm~\ref{alg_part} requires 3.742~sec. The cover time for the paths presented in Fig.~\ref{fig:opt_alg_part_pranay_grid} is 27. In Table~\ref{part_MDP_results}, we present the expected cover time of the optimal partition and the partition generated by Algorithm~\ref{alg_part} for the scenarios in Fig.~\ref{fig:sim_part_all} and Fig.~\ref{fig:sim_part_all_impl}. We also implement Algorithm~\ref{alg_heur_ct} as described in Section~\ref{num_mdp_multi} for multiple agents and compare the respective runtime. From Table~\ref{part_MDP_results}, Algorithm~\ref{alg_part} is $10^4$ times faster than the naive procedure, but generating the optimal partition for the gridworld scenarios of Fig.~\ref{fig:sim_part_all} and Fig.~\ref{fig:opt_alg_part_pranay_grid}.  We note that the average cover time of Algorithm~\ref{alg_heur_ct} on the heuristic partition generated by Algorithm~\ref{alg_part} is better than the average cover time of Algorithm~\ref{alg_heur_ct} on the optimal partition for the scenario in Fig.~\ref{fig:opt_alg_part_pranay_grid}. This could be because the suboptimality of Algorithm~\ref{alg_heur_ct} on the optimal partition might be worse than the suboptimality of Algorithm~\ref{alg_heur_ct} on the partition generated by Algorithm~\ref{alg_part} for the planning mission in Fig.~\ref{fig:opt_alg_part_pranay_grid}. Algorithm~\ref{alg_heur_ct} has an average suboptimality of $24~\%$ for the gridworld environments in {Table~\ref{part_MDP_results} when compared to a $50~\%$ suboptimality for the random MDPs in Table~\ref{part_MDP_results}. Most of the real-world environments have sparse dynamics~\cite{oceans, stoch_traffic} like the gridworld environment, and our algorithms have a lower suboptimality for such scenarios when compared to random MDPs. In particular, for the scenario in Fig.~\ref{fig:opt_alg_part_pranay_grid}, the average cover time using Algorithm~\ref{alg_heur_ct} for $m=3$ agents is only $5\%$ greater than the optimal expected cover time. This suboptimality is offset by a considerable reduction in runtime as given in Table~\ref{part_MDP_results}.

\bgroup
\def\arraystretch{1}
\centering
\begin{table*}[!b]
 \captionsetup{justification=centering}
\caption{Cover time comparison on the random MDPs presented in Table~\ref{MDP_results} with $m=3$ agents, and the gridworld environments in Fig.~\ref{fig:sim_part_all} and Fig.~\ref{fig:sim_part_all_impl}.}
\centering
\begin{tabular}{||p{0.4cm}|p{0.25cm}|p{2cm}|p{1.1cm}|p{2cm}|p{1.1cm}|p{2cm}|p{1cm}|p{2cm}|p{1cm}||}
\hline
\multicolumn{2}{||c|}{MDP} & \multicolumn{2}{|p{3.1cm}|}{Optimal single-agent algorithm on optimal partition $\mathcal{P}^*$} & \multicolumn{2}{|p{3.1cm}|}{Optimal single-agent algorithm on partition $\mathcal{P}$ by Algorithm~\ref{alg_part}} & \multicolumn{2}{|p{3cm}|}{Algorithm~\ref{alg_heur_ct} on optimal partition $\mathcal{P}^*$} & \multicolumn{2}{|p{3cm}||}{Algorithm~\ref{alg_heur_ct} on partition~$\mathcal{P}$ by Algorithm~\ref{alg_part}}\\
\hline
$|S|$ & $|\mathcal{V}|$ & Expected cover time $\max_{i} \mathbb{E}^{\pi_i^*}\left[C^{P_i}_{s_0}\right]$ & \textit{Runtime (sec)} & Expected cover time $\max_{i} \mathbb{E}^{\pi_i^*}\left[C^{P_i}_{s_0}\right]$ & \textit{Runtime (sec)} & Average cover time & \textit{Average runtime (sec)} & 
Average cover time & \textit{Average runtime (sec)} \\
\hline
\hline
50  & 10 & 17.77 &  1760.475 & 19.293 & 1573.819 & 26.749  & 186.796 & 27.754 & 0.069\\
\hline
70  & 8 & 23.074 & 407.727 & 23.074 & 395.649 & 34.622 & 12.176 & 35.663 &0.083 \\
\hline
80  & 10 & 17.6	& 4357.343 & 17.63 & 4166.866 & 26.818 &  190.694 & 26.766 & 0.196 \\
\hline
100  & 9 & 24.937 & 1927.184 & 25.094 & 1850.321 & 38.63	&  77.076 & 38.906 & 0.163 \\
\hline
100  & 10 & 25.453 & 6672.765 & 25.453	& 6400.207 & 38.575	&  272.768 & 39.382 & 0.183 \\
\hline
200  & 10 & 32.572 & 59529.463 & 34.491 & 59338.545 & 48.773 &  191.459 & 50.725 & 0.478 \\
\hline
200  & 9 & 28.8825 & 3889.93 & 29.523 & 3862.914 & 41.333 &  27.497 & 41.316 & 0.424\\
\hline
150  & 8 & 28.882 & 745.099 & 29.523 & 741.047 & 46.612 &  4.273 & 47.933 & 0.201\\
\hline
170  & 10 & 34.5429 & 43801.439 & 35.206 & 43615.032 & 54.621 & 186.73 & 54.488 & 0.284 \\
\hline
120  & 10 & 27.076 & 14830.499 & 27.907 & 14654.199 & 42.301 &  176.576 & 43.091 & 0.231\\
\hline
1000 & 9 & 35.175 & 10693.451 & 36.798 & 10511.28 & 49.25 & 51.982 & 51.02 & 2.364\\
\hline
500  & 10 & N/A & timeout & N/A & timeout &  N/A & timeout & 62.953 & 7.67\\
\hline
500  & 50 & N/A & timeout &  N/A & timeout &  N/A & timeout & 351.067 & 30.698\\
\hline
1000  & 12 & N/A & timeout & N/A & timeout &  N/A & timeout & 292.49 & 20.12\\
\hline
\multicolumn{2}{||c|}{Fig.~\ref{fig:opt_alg_part_10grid_9T}} & 19.685 &  3437 & 19.685 & 2027.843 & 25.601 & 1410.286 & 25.68 & 1.129\\
\hline
\multicolumn{2}{||c|}{Fig.~\ref{fig:opt_alg_part_10grid_10T}} & 23.562 & 4160.2 & 23.562 & 2653.843 & 31.325 & 1507.505 & 32.057 & 1.148 \\
\hline
 \multicolumn{2}{||c|}{Fig.~\ref{fig:opt_alg_part_pranay_grid}} & 29.947 & 7346 & 29.947 & 5762.169 &  32.007 & 1589.528 & 31.63 & 5.86 \\
 \hline
  \multicolumn{2}{||c|}{Fig.~\ref{fig:sim_part_heur}} & N/A & timeout & N/A & timeout &  N/A & timeout & 53.52 & 16.82 \\
 \hline
\end{tabular}
\label{part_MDP_results}
\end{table*}
\egroup

In Fig.~\ref{fig:sim_part_heur}, we depict a scenario with $|\mathcal{V}|=40$ clustered target states and $m=4$ agents.  Since some clusters are farther from the initial state than other clusters, assigning just one agent to the farthest cluster and one other agent to the nearest cluster might not be optimal. Accordingly, the target assignment obtained by Algorithm~\ref{alg_part} is not the clustered partition. Two agents (brown and violet) are assigned some targets in the cluster closer to the initial state, along with some targets in clusters that are further away. The cover time results given below validate our argument that the clustered partition might not be optimal for the scenario in Fig.~\ref{fig:sim_part_heur}. The cover time for the paths presented in Fig.~\ref{fig:sim_part_heur} is 48. The average cover time for 100 runs of Algorithm~\ref{alg_heur_ct} for the partition in Fig.~\ref{fig:sim_part_heur} is 53.52 and the variance is $18.832$. The average cover time for $100$ runs of Algorithm~1 for the partition with natural clustering is $61.97$ and the variance is $17.821$. The average runtime for 100 runs of the complete mission to jointly visit the targets by multiple agents is 16.82 sec. On the other hand, it is infeasible to solve Problem~\ref{probl_multi_ct} by a brute force procedure. The size of the MDP and target set in Fig.~\ref{fig:sim_part_heur} requires $1.0995\times~10^{12}$ expected cover time values to be computed at every step of the optimal policy iteration procedure to solve Problem~\ref{probl_ct}. Then, all the $1.0995\times~10^{12}$ expected cover time values and the corresponding subsets of the target set should be searched for the optimal partition that solves Problem~\ref{probl_multi_ct}. 

\section{Conclusion and Future work}
\label{sec:conclus}

In this paper, we formulated a path planning problem for a team of agents to jointly visit a set of target states in minimum expected time on a Markov decision process. We showed that the decision version of the planning problem on a graph for the single-agent case is the Hamiltonian path problem which is NP-complete. We also showed that the path planning problem with multiple targets is a SSP problem on a product MDP and thus policy iteration can be adopted as the optimal algorithm that is exponential in the number of target states. We utilized policy iteration as a motivation to propose a suboptimal algorithm based on an approximate value function which consumes polynomial number of operations at each time step. We proved that our heuristic algorithm eventually visits the target states with probability 1, and generates optimal solutions for certain classes of deterministic MDPs, i.e., graphs. 


For the multi-agent case, we formulated a problem of assigning multiple targets to multiple agents such that the largest optimal expected cover time among the multiple agents is minimal. The average length of a Hamiltonian path on a model graph was used as a heuristic to approximate the optimal cover time. We adopted a heuristic partitioning procedure that attempts to solve the $m$-TSP on the model graph, which also solves our multi-agent problem on the original MDP. We proved that the heuristic procedure generates optimal partitions for clustered target states. We validated our heuristic procedures on more general MDPs, as well as gridworld environments motivated by realistic considerations of ocean currents.


In subsequent work we aim to provide theoretical guarantees on the suboptimality of our heuristic algorithms. An interesting direction is to find the relationship between jointly optimizing for the policies of multiple agents and implementing the optimal policy for each agent using the optimal partition. One immediate extension of our work would be to derive conditions of optimality with our heuristic algorithm for clusters with different number of targets. It would be interesting to explore how the initial partition could affect the partition generated by our heuristic procedure. The optimality of equal sized initial partition when the hitting times are assumed to be uniformly distributed is another interesting theoretical question. We think an analysis in the decrease in cover time with increasing number of agents is an interesting direction. We also see a need for developing a unified framework to integrate our partitioning algorithm and the path planning heuristic for complex real-life problems. Such problems could have different features in the optimization objective such as resilience, dynamic travel time and different initial state for each agent.

\section*{Acknowledgment}

We thank Pranay Thangeda for helping us generate stochastic gridworlds using models of ocean currents. We also thank Anakin Dey for helping us review the writing of the manuscript.     



\bibliographystyle{ieeetr}
\bibliography{root} 

\begin{IEEEbiography}[{\includegraphics[width=1in,height=1.25in,clip,keepaspectratio]{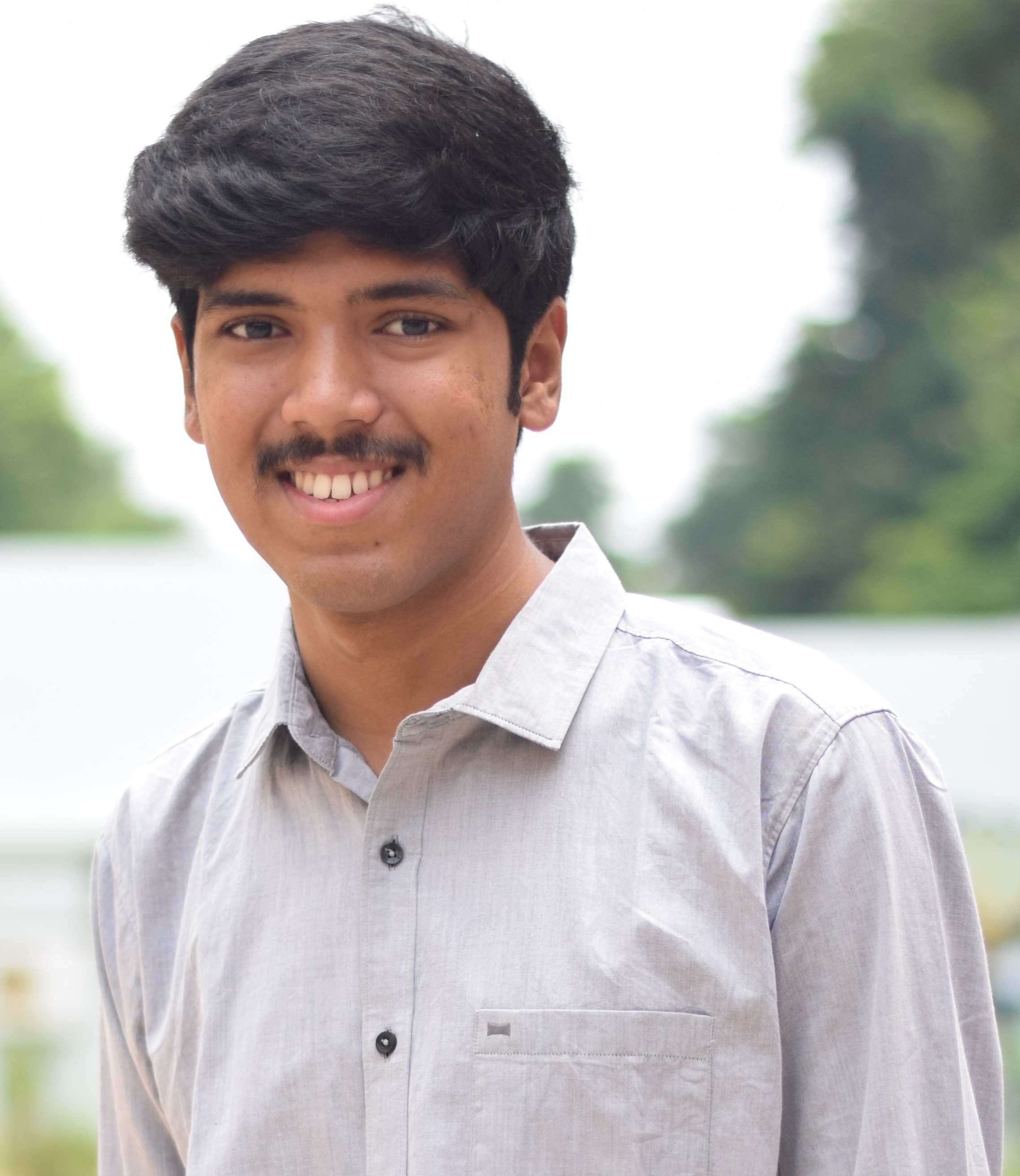}}]{Farhad Nawaz} is a Ph.D. student in Electrical and Systems Engineering at the University of Pennsylvania. He received his M.S. degree in Aerospace Engineering from the University of Illinois Urbana-Champaign in May 2021. His research interests lie in the areas of combining machine learning and control theory for problems in uncertain and complex dynamical systems. He envisions developing intelligent control frameworks for autonomous systems.
\end{IEEEbiography}

\begin{IEEEbiography}[{\includegraphics[width=1in,height=1.25in,clip,keepaspectratio]{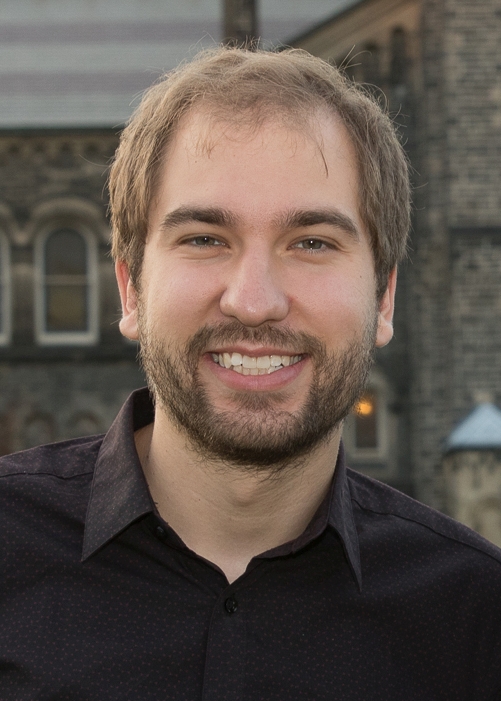}}]{Melkior Ornik} is an assistant professor in the Department of Aerospace Engineering and the Coordinated Science Laboratory at the University of Illinois Urbana-Champaign. He received his Ph.D. degree from the University of Toronto in 2017. His research focuses on developing theory and algorithms for learning and
planning of autonomous systems operating in uncertain, complex and changing environments, as well as in scenarios where only limited knowledge of the system is available.
\end{IEEEbiography}

\end{document}